\documentclass[10pt]{amsart}
\usepackage[latin1]{inputenc}  
\usepackage[T1]{fontenc}

\usepackage{amscd, amsmath, amsthm, amssymb, amsfonts}
\usepackage{graphicx}
\usepackage{mathrsfs}
\usepackage[all]{xy}
\usepackage{verbatim}
\edef\restoreparindent{\parindent=\the\parindent\relax}
\usepackage{parskip}
\restoreparindent
\usepackage[pagebackref]{hyperref}
\usepackage[dvipsnames]{xcolor}
\usepackage{braket}
\usepackage{cancel}
\usepackage[mathscr]{eucal}

\makeatletter\usepackage{microtype}\g@addto@macro\@verbatim{\microtypesetup{activate=false}}\makeatother%

\theoremstyle{plain}
\newtheorem{theorem}{Theorem}[section]
\newtheorem{lemma}[theorem]{Lemma}
\newtheorem{proposition}[theorem]{Proposition}
 
\newtheorem*{conjecture*}{Conjecture} 
\newtheorem{corollary}[theorem]{Corollary}

\theoremstyle{definition}
\newtheorem{definition}[theorem]{Definition} 
\newtheorem{example}[theorem]{Example}
\newtheorem{assumption}[theorem]{Assumption}

\theoremstyle{remark}
\newtheorem{remark}[theorem]{Remark}

\theoremstyle{plain}
\newtheorem{thm}[theorem]{Theorem}
\newtheorem*{thm*}{Th\'{e}or\`{e}me}

\newtheorem{lem}[theorem]{Lemma}
\newtheorem{pro}[theorem]{Proposition}

\newtheorem{cor}[theorem]{Corollary}
\newtheorem{conj}[theorem]{Conjecture}

\theoremstyle{definition}
\newtheorem{Def}[theorem]{Definition}

\numberwithin{equation}{section}

\newcommand\OO{{\mathcal{O}}}

\newcommand\CC{{\mathbb{C}}}

\newcommand\PP{{\mathbb{P}}}
\newcommand\QQ{{\mathbb{Q}}}
\newcommand\RR{{\mathbb{R}}}
\newcommand\ZZ{{\mathbb{Z}}}

\newcommand\Amp{{\rm Amp}} 
 \newcommand\Aut{{\rm Aut}} 
 
\newcommand\Bigc{{\rm Big}} 
\newcommand\Eff{{\rm Eff}} 
\newcommand\Effp{{\rm Eff}^{+}} 
\newcommand\Pseff{\overline{{\rm Eff}}}

\newcommand\Nef{{\rm Nef}}
\newcommand\Nefe{{\rm Nef}^e} \newcommand\Nefp{{\rm Nef}^{+}} 

\newcommand\Mov{\overline{{\rm Mov}}}
\newcommand\Move{\overline{{\rm Mov}}^e}

\newcommand\Movp{{\rm Mov}^{+}}
\newcommand\Movo{{\rm Mov}^{\circ}}
\newcommand\Pic{\text{\rm Pic}}

\newcommand\GL{{\rm GL}} 

\newcommand\id{{\rm id}}

\newcommand{\ssec}{\subsection}

\newcommand{\ol}{\overline}

\newcommand{\vast}{\bBigg@{4}}
\newcommand{\Vast}{\bBigg@{5}}
\newcommand{\wt}{\widetilde}

\newcommand{\cC}{\mathcal{C}}

\newcommand{\cM}{\mathcal{M}}
\newcommand{\cN}{\mathcal{N}}
\newcommand{\cO}{\mathcal{O}}
\newcommand{\cP}{\mathcal{P}}

\newcommand{\fF}{\mathfrak{F}}

\newcommand{\gD}{\Delta}
\newcommand{\gG}{\Gamma}

\newcommand{\gS}{\Sigma}

\newcommand{\ga}{\alpha}
\newcommand{\gb}{\beta}

\newcommand{\Alb}{\mathrm{Alb}}

\newcommand{\Cone}{\mathrm{Cone}}

\newcommand{\Div}{\mathrm{Div}}

\newcommand{\Exc}{\mathrm{Exc}}

\newcommand{\op}

\newcommand{\PsAut}{\mathrm{PsAut}}

\newcommand{\Stab}{\mathrm{Stab}}

\newcommand{\Supp}{\mathrm{Supp}\,}

\newcommand{\cnec}{\mathrel{:=}}

\newcommand{\dto}{\dashrightarrow}

\newcommand*\eto{%
	\xrightarrow[]{\raisebox{-0.25 em}{\smash{\ensuremath{\sim}}}}%
}

\usepackage{mathabx,epsfig}
\def\acts{\mathrel{\reflectbox{$\righttoleftarrow$}}}

\begin{document}
	
	\title[The effective cone conjecture]{The effective cone conjecture for Calabi--Yau pairs}

	\author{C\'ecile Gachet}
    \address{Fakult\"at f\"ur Mathematik, Ruhr-Universit\"at Bochum, Universit\"atsstr. 150, Postfach IB 45, 44801 Bochum, Germany.}
	\email{cecile.gachet@rub.de}

	\author{Hsueh-Yung Lin}
	\address{Department of Mathematics, National Taiwan University, 
	and National Center for Theoretical Sciences,
	Taipei, Taiwan.}
	\email{hsuehyunglin@ntu.edu.tw}

	\author{Isabel Stenger}
	\address{Institute of Algebraic Geometry, Leibniz University Hannover, Welfengarten 1, 30167 Hannover , Germany.}
	\email{stenger@math.uni-hannover.de}
	
	\author{Long Wang}
	\address{Center for Mathematics and Interdisciplinary Sciences, Fudan University, Shanghai, 200433, China; Shanghai Institute for Mathematics and Interdisciplinary Sciences, Shanghai, 200433, China.}
	\email{wanglll@fudan.edu.cn}

	\begin{abstract}
 We formulate an {\it effective cone conjecture} for klt Calabi--Yau pairs $(X,\Delta)$ describing the structure of the cone of effective divisors $\Eff(X)$ modulo the action of the subgroup of pseudo-automorphisms $\PsAut(X,\Delta)$. Assuming the existence of good minimal models in dimension $ \dim X$, 
 known to hold in dimension up to $3$, 
 we prove that the effective cone conjecture for $(X,\Delta)$
 is equivalent to the Kawamata--Morrison--Totaro movable cone conjecture for $(X,\Delta)$, 
 among other statements.
 As an application, 
 we show that the movable cone conjecture unconditionally holds
 for the smooth Calabi--Yau threefolds introduced by Schoen and studied by Namikawa, Grassi and Morrison.
 We also show that for such a Calabi--Yau threefold $X$,  all of its minimal models, 
 apart from $X$ itself, have rational polyhedral nef cones. 
\end{abstract}

\maketitle

\section{Introduction}\label{sec_introduction}

The starting point of this paper is the following conjecture.

\begin{conjecture*}[Kawamata--Morrison cone conjecture]\label{conj-KMconeVar}
    Let $X$ be a normal $\QQ$-factorial terminal Calabi--Yau variety. Then
    \begin{enumerate}
    \setlength{\itemsep}{10pt}
        \item There is a rational polyhedral cone $\Pi\subset\Nef(X)$ such that 
    $$\Aut(X)\cdot \Pi = \Nefe(X) = \Nefp(X),$$ 
    and for every $g\in\Aut(X)$, $g^*{\Pi}^{\circ}\cap{\Pi}^{\circ}\neq\emptyset$ if and only if $g^*=\id$.
        \item There is a rational polyhedral cone $\Sigma\subset\Mov(X)$ such that 
    $$\PsAut(X)\cdot \Sigma = \Move(X) = \Movp(X),$$ 
    and for every $g\in \PsAut(X)$, $g^*{\Sigma}^{\circ}\cap{\Sigma}^{\circ} \neq\emptyset$ if and only if $g^*=\id$.
    \end{enumerate}
\end{conjecture*}

\noindent Here, for a convex cone $\mathcal{C}$ in the N\'eron--Severi space $N^1(X)_{\RR}$,  we denote by
$\mathcal{C}^e$ the intersection $\cC\cap\Eff(X)$, and by
$\mathcal{C}^+$ the convex cone spanned by the classes of Cartier divisors in $\ol{\mathcal{C}}$. We also denote by $\PsAut(X)$ the group of pseudo-automorphisms of $X$, i.e., birational self-maps of $X$ which are isomorphisms in codimension $1$.

This conjecture has appeared in various forms, notably stated by Morrison, Kawamata, and Totaro in \cite{Mo93,Mo96,Ka97,To08,To10} (in order of increasing generality).  Albeit initially motivated by mirror symmetry, it has attracted much work from birational geometers over the past thirty years, see \cite{LOP18} for a survey.

The main purpose of our work is to understand the role of the cone of effective divisors $\Eff(X)\subset N^1(X)_{\RR}$ within this conjectural picture. The importance of the effective cone, or more generally of its subcones, was first noticed and exploited by \cite{LZ22}. Their main conjecture and theorem inspired us to introduce an effective cone conjecture for klt Calabi--Yau pairs and relate it to the movable cone conjecture.
Assuming the existence of good minimal models, we prove the equivalence of the effective and the movable cone conjectures, and establish relations to other versions of the cone conjecture.
This constitutes Theorem~\ref{main_thm-equiv}, the main result of this paper.

As an application in Theorem \ref{thm_mainSchoen}, we settle the movable cone conjecture for the smooth Calabi--Yau threefolds studied in \cite{Sc88,Na91,GM93}.

 \subsection{A chamber decomposition of the effective cone}\label{ssec-state-eff-dec}

 We first describe the birational geometric aspects of Calabi--Yau pairs encoded in the effective cone.
    In~\cite{Ka97}, Kawamata proved that for a Calabi--Yau threefold $X$,
    \begin{equation}\label{decomp-Move}
        \Move(X) = \bigcup_{\substack{\alpha: X\dto X' \\\mbox{\footnotesize{{\rm SQM}}}}}
    \alpha^*\Nefe(X');
    \end{equation} 
    where $\{\alpha\colon X\dto X'\}$ ranges through all small $\QQ$-factorial modifications (or SQM for short, see \S\ref{sssec-BC}) of $X$.
    Assuming the existence of minimal models for $X$ as stated in Definition~\ref{def_minimal_model}(3),
    this result generalizes to any variety $X$ which underlies a klt Calabi--Yau pair
    (see e.g. \cite[Theorem 3.5]{SX23}).

Under the same assumption, 
we prove that Decomposition~\eqref{decomp-Move} can be extended to a similar
chamber decomposition of the effective cone $\Eff(X)$.
To define the chambers,
we reintroduce a notion originating from the influential work \cite{HK00} on Mori dream spaces.  For any $\QQ$-factorial birational contraction (or QBC for short, see \S\ref{sssec-BC}) $f\colon X \dashrightarrow Y$, we define the {\it $f$-Mori chamber} to be
$$\Eff(X;f) \cnec 
    f^*\Nefe(Y)+\sum_{E\in \Exc(f)}\RR_{\ge 0}[E] \quad \subset \quad \Eff(X),$$
    where $\Exc(f)$ denotes the set of prime exceptional divisors of $f$.

  \begin{pro}\label{pro-EffCD}
    Let $(X,\gD)$ be a klt Calabi--Yau pair.
    Assume the existence of minimal models for $X$ as in Definition~\ref{def_minimal_model}(3).
    We have a chamber decomposition
        $$\Eff(X)=
        \bigcup_{\substack{f:X\dto Y\\ \mbox{\footnotesize{{\rm QBC  of} X}}}}\Eff(X;f).
        $$
    \end{pro}

An expanded statement for a decomposition  of the interior of the effective cone can be found in Proposition~\ref{prop_decompeffmov}.

\subsection{The effective cone conjecture, and other cone conjectures}\label{ssec-state-eff-cc} 

Let $X$ be a normal $\QQ$-factorial projective variety. For a subgroup $G$ of $\PsAut(X)$, we denote by $G^{\ast}$ its image under the natural action 
by pullback (see Subsection~\ref{subssec-pfwd})
$$g\in \PsAut(X) \mapsto g^*\in \GL\left( N^1(X)_{\RR} \right).$$

For a pair $(X,\Delta)$, we follow Totaro \cite{To10} and introduce the groups 

\begin{align*}
 \Aut(X,\Delta) & = \Set{ g \in \Aut(X) : g(\Supp \Delta) = \Supp \Delta}  \text{ and }  \\
 \PsAut(X,\Delta) &= \Set{ g \in \PsAut(X) : g(\Supp \Delta) = \Supp \Delta}.
\end{align*}

\begin{definition}\label{def-cone-conjs}
    We say that the {\it movable cone conjecture}, the {\it effective cone conjecture}, the {\it nef cone conjecture}, respectively the {\it $f$-Mori chamber cone conjecture},  holds for the pair $(X,\Delta)$ (with its $\QQ$-factorial birational contraction $f:(X,\Delta)\dto (Y,\Delta_Y)$) if there exists a rational polyhedral fundamental domain for the action 
    \begin{enumerate}
        \item (Movable cone conjecture) $\PsAut^{\ast}(X,\Delta) \acts \Move(X)$
        \item (Effective cone conjecture) $\PsAut^{\ast}(X,\Delta) \acts \Eff(X)$, 
        \item (Nef cone conjecture) $\Aut^{\ast}(X,\Delta) \acts \Nefe(X)$, 
        \item ($f$-Mori chamber cone conjecture) $\PsAut^{\ast}(X,\Delta;f) \acts \Eff(X;f)$.
    \end{enumerate}

\noindent  Here, the group $\PsAut^*(X,\Delta;f)$ is the stabilizer of the $f$-Mori chamber $\Eff(X;f)$ in the group $\PsAut^*(X,\Delta)$; see Subsection \ref{ssec-stab-mori} for an alternative definition. 
\end{definition}
The \textit{Kawamata--Morrison--Totaro cone conjecture} (see \cite{To10}) predicts that (1) and (3) in Definition \ref{def-cone-conjs} hold for klt Calabi--Yau pairs. Note that this fails for {\it log canonical} Calabi--Yau pairs; see~\cite{To10}. We propose a related conjecture.

\begin{conj}\label{conj-eff}
    Let $(X,\gD)$ be a klt Calabi--Yau pair.
    The effective cone conjecture, as stated in Definition \ref{def-cone-conjs}(2), holds for $(X,\gD)$.
\end{conj}

We can now state our main result.

\begin{theorem}\label{main_thm-equiv}
Let $(X,\gD)$ be a klt Calabi--Yau pair. Assume that good minimal models exist in dimension $\dim(X)$, as stated in Definition~\ref{def_minimal_model}(4). Then the statements (1), (2), (3), and (4) below are equivalent: 
        \begin{enumerate}
            \item The movable cone conjecture holds for $(X,\gD)$.
            \item The effective cone conjecture holds for $(X,\gD)$. 
            \item 
            \begin{enumerate}
            \item The nef cone conjecture holds for each klt pair $(X',\gD')$ obtained by a small $\QQ$-factorial modification from $(X,\gD)$.
            \item There are only finitely many $(X', \Delta')$ up to isomorphism of pairs arising as small $\QQ$-factorial modifications of $(X,\Delta)$.
            \end{enumerate}
            \item 
            \begin{enumerate}
            \item The Mori chamber cone conjecture holds for each 
            $\QQ$-factorial birational contraction $f\colon (X,\gD)\dto (Y,\gD_Y)$.
            \item There are only finitely many $(Y,\Delta_Y)$ up to isomorphism of pairs arising as $\QQ$-factorial birational contractions of $(X,\Delta)$.
            \end{enumerate}
        \end{enumerate} 
    \end{theorem} 

A more detailed and general formulation appears later as Theorem~\ref{thm_equivalence_detailed} and highlights where the existence of good minimal models is necessary.

 \subsection{A descent result for the nef cone conjecture}\label{ssec-state-descent-nef}
    In  proving the implication $(3) \Rightarrow (4)$ in Theorem~\ref{main_thm-equiv},
    we show a {\it birational descent} result
    for the nef cone conjecture which is a generalization of \cite[Lemma 3.4]{To10} to higher dimension.
    \begin{proposition}\label{prop_descentIntro}
    Let $(X,\Delta)$ be a klt Calabi--Yau pair.
    Assume that the nef cone conjecture holds for $(X,\Delta)$. 
    Then, for any birational morphism $f\colon(X,\Delta)\to (Y,\Delta_Y)$ 
    with $Y$ being $\QQ$-factorial, 
    the nef cone conjecture holds for $(Y,\Delta_Y)$.
    \end{proposition}

We will give a slightly more general 
result in Proposition~\ref{prop_descent}.

\subsection{Contextualizing Theorem~\ref{main_thm-equiv}
and relation to other works}\label{ssec-remknwon}

All the numbers in this subsection 
correspond to the statements in Theorem~\ref{main_thm-equiv}.

The implication $(1)\Rightarrow (3)$
was asked by
Oguiso in~\cite[Question 2.30(2)]{Og22}. 
While  $(1) \Rightarrow \text{(3b)}$ is proven 
in \cite[Theorem 2.14]{CL14} and \cite[Proposition 5.3]{LZ22} under the assumption of the existence of good minimal models, 
the implication $(1)\Rightarrow \text{(3a)}$ was still open when we started to work on Theorem \ref{main_thm-equiv}. Shortly before the current version of this paper was ready to be made public, the preprint~\cite{xu2024cone} by F. Xu appeared on the arXiv.
Independently, the author proves the implication 
(1) $\Rightarrow \text{(3a)}$
under the assumption of the existence of minimal models, and a non-vanishing assumption (see~\cite[Theorem 14]{xu2024cone}).

Our implication [(1) or (2)] $\Rightarrow$ (3a)
is also related to~\cite[Conjecture 1.2(2), Theorem 1.3(3)]{LZ22}
in the work by Z. Li and H. Zhao. 
We also refer the reader to Statement $(2')$ in Theorem \ref{thm_equivalence_detailed}.

The implication $(2) \Rightarrow (1)$ under the assumption of the existence of good minimal models is an immediate consequence of \cite[Theorem 1.3(1)]{LZ22}.

Finally in the same preprint~\cite{xu2024cone}, 
F. Xu proves a reduction result for the movable cone conjecture in~\cite[Theorem 2]{xu2024cone}, of which one direction is the birational descent for the movable cone conjecture. 
It is also independent from our descent result for the nef cone conjecture (Proposition \ref{prop_descentIntro}).

\subsection{Cone conjectures for the Calabi--Yau threefolds introduced by Schoen}\label{ssec-state-application-schoen}

The equivalence of the statements in Theorem \ref{main_thm-equiv} offers some flexibility to study one of the cone conjectures, and derive results for the others.
For instance, if a klt Calabi--Yau pair $(X,\Delta)$ of dimension 3 satisfies the effective cone conjecture, then by Theorem \ref{main_thm-equiv} (and the existence of good minimal models in dimension 3), it also satisfies the nef and movable cone conjectures.
This principle works well for the following class of smooth Calabi--Yau threefolds $X$ introduced by Schoen in~\cite{Sc88}.

    \begin{thm}\label{thm_mainSchoen} 
    Let $X$ be a smooth 
    Calabi--Yau threefold obtained as a fiber product 
    $W_1\times_{\PP^1}W_2$, where $\phi_i\colon W_i\to\PP^1$ is a relatively minimal rational elliptic surface with a section for $i=1,2$. 
    Assume moreover that the generic fibers of $\phi_1$ and $\phi_2$ are non-isogenous.
    Then all statements of Theorem \ref{main_thm-equiv} hold for $X$.
	\end{thm}  

Together with \cite{Na91}, Theorem \ref{thm_mainSchoen} provides the following description of the nef cones of minimal models of $X$.

 \begin{corollary}\label{cor_ratpol}
Keep the notation as in Theorem \ref{thm_mainSchoen}. Assume further that all singular fibers of $\phi_1$ and $\phi_2$ are irreducible nodal rational curves.
Then, every minimal model $X'$ of $X$ which is not isomorphic to $X$ has a rational polyhedral nef cone.
 \end{corollary}

 The nef cone conjecture was established for such $X$ by Grassi--Morrison in \cite{GM93}. 
 The finiteness of minimal models was proven by Namikawa in \cite{Na91}, who even provided an explicit count of the (over $5\times 10^{19}$) minimal models. To our knowledge, the movable cone conjecture for $X$ and the nef cone conjecture for its other minimal models are both new results.

   \subsection*{Organization of the paper}  
   Section \ref{sec_preliminaries} reviews preliminaries. In Section \ref{sec-convgeom}, we prove preparatory results in convex geometry in the spirit of Looijenga \cite{Lo14}. In Section \ref{sec-chamber-decs}, we introduce Mori chambers, recall results on the geography of models, and prove various decompositions of cones. Section \ref{sec-descent} proves our descent result for the nef cone conjecture. Section \ref{sec-PfThMain} focuses on the proof of Theorem \ref{thm_equivalence_detailed}, which is a detailed version of Theorem \ref{main_thm-equiv}. Finally, Section \ref{sec-application} is devoted to our results on the smooth Calabi--Yau threefolds introduced by Schoen.

   \subsection*{Acknowledgments} We would like to thank Chen Jiang and Vladimir Lazi\'c for valuable discussions on the minimal model program, and for suggesting some relevant references. We also thank Keiji Oguiso for insightful questions and helpful discussions, and Jungkai Chen, Gavril Farkas, Andreas H\"oring, Ching-Jui Lai, Gianluca Pacienza, Zhixin Xie and Xun Yu for various comments and encouragements. We are grateful to the referees for constructive suggestions and pointing out several mistakes.

   We acknowledge the following support. Gachet was partly supported by the ERC Advanced Grant SYZYGY. This project has received funding from the European Research Council (ERC) under the European Union Horizon 2020 research and innovation program (grant agreement No. 834172). Lin is supported by 
   the Yushan Fellow Program by the Ministry of Education (NTU-114V1006-5), the National Science and Technology Council (114-2628-M-002-012-, 114-2639-M-002-009-ASP) in Taiwan,
   and the Asian Young Scientist Fellowship. Wang is supported by the National Key Research and Development Program of China (No. 2023YFA1010600), the National Natural Science Foundation of China (No. 12401052), the NSFC for Innovative Research Groups (No. 12121001), and the Postdoctoral Fellowship Program of CPSF (No. GZC20230535).

    \section{Notation and preliminaries}\label{sec_preliminaries} 
    
    We work over the field $\CC$ of complex numbers. 
    We refer to ~\cite{KM98} for standard results of birational geometry, and to \cite{Fu17} for results and definitions regarding $\RR$-divisors more specifically.
    
    In this paper, a \textit{pair} is the data $(X,\Delta)$ of a normal $\QQ$-factorial projective variety\footnote{We add $\QQ$-factoriality for convenience.} $X$ and of an effective $\RR$-divisor $\Delta$ on $X$. 
    We call a pair $(X,\Delta)$ \textit{Calabi--Yau} if the $\RR$-divisor $K_X+\Delta$ is numerically trivial, following \cite{To10}. Many pairs considered in this paper are \textit{klt}; for a definition, see \cite[Definition 2.3.4]{Fu17}.
   
    \subsection{Cones of numerical classes of divisors.}\label{ssec-Defcones}
    Let $X$ be a normal projective variety. We write $N^1(X)$ for the free abelian group generated by the classes of Cartier divisors modulo numerical equivalence.
    Inside the vector space $N^1(X)_{\RR}:=N^1(X)\otimes \RR$, we denote by $\Nef(X)$ the \textit{nef cone}, i.e., the closure of the ample cone $\Amp(X)$, and by $\Eff(X)$ the \textit{effective cone}, that is, the cone generated by the numerical classes of effective Cartier divisors in $N^1(X)_{\RR}$. The closure and the interior of $\Eff(X)$ are called the {\it pseudo-effective cone} $\Pseff(X)$ and the {\it big cone} $\Bigc(X)$ respectively; 
    note that they need not equal $\Eff(X)$ in general.
    The \textit{nef effective cone} $\Nefe(X)$ is defined as
    \[ \Nefe(X) := \Nef(X)\cap \Eff(X). \] 
      
    A Cartier divisor $D$ on a projective variety $X$ is called \textit{movable} if there is a positive integer $m$ such that $mD$ is effective and the base locus of the linear system $|\cO_X(mD)|$ does not contain any divisor. We denote by $\mathrm{Mov}(X)$ the convex cone in $N^1(X)_{\RR}$ generated by the numerical classes of movable Cartier divisors. 
    In general, the cone $\mathrm{Mov}(X)$ is neither open nor closed. The {\it closed movable cone} $\Mov(X)$ and the {\it open movable cone} $\Movo(X)$ are the closure, respectively, the interior of $\mathrm{Mov}(X)$. The \textit{movable effective cone} $\Move(X)$ is defined as
    \[ \Move(X) := \Mov(X)\cap \Eff(X). \]

 We have the following inclusions of cones

\begin{center}
$\Nef(X) \subset \Mov(X) \subset \Pseff(X)$\\
$\Amp(X) \subset \Movo(X) \subset \Bigc(X)$.
\end{center}
In particular, each of the cones described here is strictly convex (see Subsection \ref{ssec-notation} for definition) since $\Pseff(X)$ is.
Note that if $X$ is a surface, we have
\begin{center}
$\Nef(X) = \Mov(X)$,\quad $\Amp(X) = \Movo(X)$.
\end{center}
    
    \subsection{Some group actions on cones of divisors}

\subsubsection{Birational contractions and pseudo-automorphisms}\label{sssec-BC}
Let $X$ be a normal $\QQ$-factorial projective variety.
We define some important types of birational maps. Following \cite{HK00}, we call a birational map $f\colon X \dashrightarrow Y$ to a normal projective variety $Y$ a {\it birational contraction}, if $f^{-1}$ contracts no divisor. If in addition, $Y$ is $\QQ$-factorial, we call $f$ {\it a $\QQ$-factorial birational contraction} (or {\it QBC} for short). We refer to the data of $(Y, f)$ as a {\it marked} QBC from $X$, with {\it marking} $f$.  We consider two marked QBCs $(Y_1,f_1)$ and $(Y_2,f_2)$ from $X$ to be {\it isomorphic} and write $(Y_1,f_1)\simeq (Y_2,f_2)$, if the birational map $f_2\circ f_1^{-1}$ is an isomorphism.

Following \cite[Definition 1.8]{HK00}, we define a \textit{small $\QQ$-factorial modification} (or {\it SQM} for short) of $X$ as a $\QQ$-factorial birational contraction $\alpha\colon X \dashrightarrow X'$ which contracts no divisor.
The birational maps from $X$ to itself which are
isomorphisms in codimension one are called the \textit{pseudo-automorphisms} 
of $X$, and they form a group denoted by $\PsAut(X)$. Regarding SQMs as special cases of QBCs, we obtain the notion of marked  SQMs and isomorphisms of marked SQMs.

The following example shows that there can be infinitely many marked small $\QQ$-factorial modifications, but only finitely many isomorphism classes of targets. 

\begin{example}
Let $X$ be a very general hypersurface of multidegree $(2, \dots, 2)$ in $(\PP^1)^{n+1}$ with $n \geq 3$. Then $X$ is a simply connected smooth Calabi--Yau manifold of dimension $n$ (see, e.g., \cite[Theorem 3.1]{CO15}). Every birational self-map $f$ of $X$ is a pseudo-automorphism, and gives a small $\QQ$-factorial modification $(X, f\colon X \dashrightarrow X)$. 
By \cite[Theorem 3.3]{CO15}, the automorphism group $\Aut(X)$ is trivial, while the pseudo-automorphism group $\PsAut(X)$ is infinite. Moreover, by the proof of \cite[Theorem 3.3(4)]{CO15}, for every small $\QQ$-factorial modification $\alpha \colon X \dashrightarrow X'$, we have $X' \cong X$. So there are infinitely many marked small $\QQ$-factorial modifications $(X, f)$, but they all have the same target variety $X$.  
\end{example}

\subsubsection{Pushforwards and pullbacks}\label{subssec-pfwd}
We define notions of pushforwards and pullbacks for real divisor classes under birational maps. Let $f\colon X\dashrightarrow X'$ be a birational map between $\QQ$-factorial varieties.

We start by defining a pullback by $f$. We resolve

\begin{center}
\xymatrixrowsep{12pt}
\xymatrixcolsep{12pt}

\begin{minipage}{0.9\textwidth}
\xymatrix{
& W \ar[ld]_{p} \ar[rd]^{q} &\\
X \ar@{-->}[rr]_{f} & & X'
}
\end{minipage}
\end{center}

\noindent with $W$ a normal $\QQ$-factorial projective variety, and $p,q$ birational morphisms. This yields a pullback group homomorphism 
$$f^{*}:=p_*q^*\colon N^1(X')_{\QQ}\to N^1(X)_{\QQ}.$$ 
The morphism $f^*$ is independent of the choice of the resolution $(W,p,q)$. 

 From now on, we assume that $f$ is a birational contraction.

We can then define a pushforward: We have a group homomorphism between the groups of codimension-one cycles
$$f_*:Z^1(X)\to Z^1(X').$$
As a consequence of the negativity lemma \cite[Lemma 3.39]{KM98} the kernel of the pushforward $f_*$ (at the level of $\QQ$-divisors) is spanned by the set of prime exceptional divisors $\Exc(f)$.

The following simple, yet important fact, relates pushforwards and pullbacks. The pushforward $f_*$ preserves numerical equivalence, thus descends to
$$f_* : N^1(X)_{\QQ}\to N^1(X')_{\QQ},$$
which coincides with $(f^{-1})^*$. Moreover, if $(W,p,q)$ is a resolution of $f$ as above, we have
    $$f_*f^*= (f^{-1})^*(f)^* = q_*p^*p_*q^* = \id,$$ and therefore $f^*$ is injective and $f_*$ is surjective.
    
We conclude \S\ref{subssec-pfwd} with the decomposition
$$N^1(X)_{\RR} = f^*N^1(X')_{\RR}\oplus {\rm Span}_{\RR}({\rm Exc}(f)),$$
which immediately follows from the various facts above.

\subsubsection{Pullbacks, pushforwards, and composition.}

There are instances of pullback functoriality, such as the following lemma. 

\begin{lemma}\label{lem_funpull}
Let $f \colon X \dto Y$ and $g \colon Y \dto Z$ be two birational maps
between normal $\QQ$-factorial projective varieties.
Assume that $f^{-1}$ is a birational contraction.
Then
$$f^*g^* = (gf)^* \colon N^1(Z)_\QQ \to N^1(X)_\QQ.$$ 
\end{lemma}

\begin{proof}
    Choose a common resolution of $f$ and $g$
    $$
    \xymatrix{
     & W \ar[dl]_p \ar[d]_q \ar[dr]^r &   \\
     X \ar@{-->}[r]_{f} & Y \ar@{-->}[r]_{g} & Z   
    }
    $$
   where $W$ is a normal $\QQ$-factorial projective variety, and $p,q,r$ are birational morphisms with $\Exc(q)\subset\Exc(p)$.
   Note that the image of $\mathrm{id}_{N^1(W)_{\QQ}}-q^*q_*$ is contained in $\ker(q_*)$. Since $p_*\ker(q_*)=0$ by assumption, we have
\[ f^*g^* = p_*q^*q_*r^*=p_*r^* = (gf)^*. \qedhere 
\]
\end{proof}

The following corollary shows an instance of pushforward functoriality.

\begin{corollary}\label{cor_funpushf}
Let $f \colon X \dto Y$ and $g \colon Y \dto Z$ be two birational contractions
between normal $\QQ$-factorial projective varieties.
Then
$$ (gf)_* = g_*f_* \colon N^1(X)_\QQ \to N^1(Z)_\QQ.$$ 
\end{corollary}

\begin{proof}
    Apply Lemma \ref{lem_funpull} to obtain that $[(gf)^{-1}]^*=(g^{-1})^*(f^{-1})^*$, and use that both $f$ and $g$ are birational contractions to identify their inverse pullback with their pushforward.
\end{proof}

\subsubsection{The $\PsAut(X)$-action on $N^1(X)_\RR$} \label{sssec-SQM}

Throughout \S\ref{sssec-SQM}, we denote by $X$ a normal $\QQ$-factorial projective variety.

For any $\QQ$-factorial birational contraction $f\colon X\dto X'$, it holds
$$f^*N^1(X')_{\QQ} \subset N^1(X)_{\QQ},
 \ \ \ f^*\Eff(X')\subset\Eff(X), \ \ \ f_*\Mov(X)\subset\Mov(X'),
 $$
with equality if and only if $f$ is a small $\QQ$-factorial modification.
In this case, we further have
$$f^*\Mov(X') = (f^{-1})_*\Mov(X') = \Mov(X).$$

In particular, pullback defines a linear right group action of $\PsAut(X)$ on $N^1(X)_{\RR}$.
It preserves the lattice of Weil divisor classes $N^1_W(X)$. It also preserves the cones $\Mov(X)$ and $\Eff(X)$. The induced action of the automorphism group $\Aut(X)$ additionally preserves the cone $\Nef(X)$.

\subsection{Some standard conjectures of the minimal model program}\label{ssec-standardMMPconj}
Part of the main results of this paper is proven under the assumption of \textit{the existence of good minimal models} in dimension $n$. In this subsection, we recall what this assumption actually means, and provide a reference for the fact that it is satisfied for $n\le 3$. Note that it is most crucial for the following definition that we allow $\RR$-divisors.

 \begin{definition}\label{def_minimal_model} We define the following notions:
    \begin{itemize} 
    	\item[(1)] Following \cite[Definition 3.50]{KM98}), we define a \textit{minimal model} of a klt pair $(X, \Delta)$ as the data of a klt pair $(W, \Delta_W)$ and a birational contraction $\phi\colon(X,\Delta)\dashrightarrow (W,\Delta_W)$ such that 
    	\begin{itemize}
            \item[(i)]  $\phi_*\Delta = \Delta_W$ as effective $\RR$-divisors;
            \item[(ii)] the $\RR$-divisor $K_W + \Delta_W$ is nef;
    	\item[(ii)] for any prime effective Weil divisor $E\subset X$ that is contracted by $\phi$, we have the inequality of discrepancies $a(E, X, \Delta) < a(E, W, \Delta_W)$. (See \cite[Lemma 2.3.2]{Fu17} for a definition.)
    	\end{itemize}
     \item[(2)] A \textit{good minimal model} of a klt pair $(X, \Delta)$ is a minimal model $(W,\Delta_W)$ such that the $\RR$-divisor $K_W + \Delta_W$ is semiample in the sense of \cite[Definition 2.1.20]{Fu17}.
    \item[(3)] Let $X$ be a normal $\QQ$-factorial projective variety. We say that we have \textit{the existence of minimal models for $X$}, respectively \textit{the existence of good minimal models for $X$}, if for any $\RR$-divisor $\Delta$ such that the pair $(X, \Delta)$ is klt and the $\RR$-divisor $K_X+\Delta$ is effective, then the pair $(X,\Delta)$ admits a minimal model, respectively a good minimal model. 
     \item[(4)] Let $n$ be a positive integer. We say that \textit{the existence of good minimal models holds in dimension $n$} if for any klt pair $(X, \Delta)$ such that $X$ has dimension $n$ and the $\RR$-divisor $K_X + \Delta$ is effective, there exists a good minimal model for the pair $(X, \Delta)$. 
    \end{itemize}
    \end{definition} 

    The existence of minimal models is known in dimension up to $4$ \cite{Bi11}, while the existence of good minimal models is known in dimension up to $3$ by \cite{KMM94} and \cite{Sh96}. 

We also consider the following \textit{good minimal model assumption},
stated for a fix $(X,\gD)$ klt pair with $K_X+\gD$ effective and $\gD$ a $\QQ$-divisor.

\begin{assumption}\label{hyp-GMMPX} 
The following properties hold
for any effective $\QQ$-divisor $B$ on $X$ such that
    $(X,\Delta + B)$ is klt.
    \begin{enumerate}
        \item[(a)] The pair $(X,\Delta + B)$ has a good minimal model $(Y,\Delta_Y+B_Y)$.
        \end{enumerate}
        Let $F$ denote a very general fiber of the canonical model of $(Y,\Delta_Y+B_Y)$. For any effective $\QQ$-divisor $D$ on $Y$ such that $(Y,\Delta_Y+D)$ is klt,
        \begin{enumerate}
        \item[(b)] The pair $(F,(\Delta_Y+D)|_F)$ has a good minimal model.
    \end{enumerate}
\end{assumption}

\begin{remark} Here are a few cases where Assumption \ref{hyp-GMMPX} is satisfied:
\begin{itemize}
    \item if $\dim X \le 3$ (\cite{KMM94} and \cite{Sh96});
    \item if $(X,\Delta)$ is a klt Calabi--Yau pair with $\dim X = 4$ and $X$ is uniruled but not rationally connected (\cite[Corollary B]{Laz25});
    \item any projective hyperk\"ahler manifold $X$ that satisfies the \emph{hyperk\"ahler SYZ conjecture} \cite[Conjecture 1.7]{Ver10}, notably any hyperk\"ahler manifold of one of the four known types; see \cite[Corollary 1.1]{Ma17}, \cite[Corollary 1.3]{MR21}, and \cite[Theorem 2.2]{MO22}. Note that when $\dim F > 0$, $F$ is an abelian variety.
\end{itemize}
\end{remark}

The following lemma
about the existence of good minimal models was kindly explained to us by Chen Jiang.

\begin{lemma}\label{lem-GMMPind} 
    Let $n$ be a positive integer.
    Suppose that good minimal models exist in dimension $n$, then they also exist in any dimension 
    $k \le n$. 
\end{lemma}

\begin{proof}

It suffices to prove the statement for $k = n-1$. Let $(X,\gD)$ be a klt pair with $\dim X = n-1$ and $K_X + \Delta$ $\RR$-effective. Fix an elliptic curve $E$. 
Then the product pair $(X \times E, \gD \times E)$ is klt of dimension $n$, and hence has a good minimal model $(Z,\gD_Z)$ by assumption. Consider the composition
$$f\colon Z \dto X \times E \to E.$$
We claim that $f$ is a morphism. Indeed, take a desingularization $\wt{Z}$ of $Z$ resolving $f$. By the universal property of the Albanese morphism, the morphism $\wt{Z} \to E$ factorizes into $\wt{Z} \to \Alb(\wt{Z}) \to E$. Since $Z$ has rational singularities, by~\cite[Proposition 2.3]{Re83} (see also \cite[Lemma 8.1]{KawamataMMPKodim}), the Albanese map $Z \dto \Alb(\wt{Z})$ is a morphism. This shows that $f$, as the composition $Z \to \Alb(\wt{Z}) \to E$, is indeed a morphism. 
So $(Z,\gD_Z)$ is a minimal model of $(X \times E, \gD \times E)$ over $E$. The general fiber $(F,\gD_F)$ of $f$ therefore is a minimal model of $(X, \Delta)$. Since $K_Z + \gD_Z$ is semiample, so is $K_F + \gD_F$.
\end{proof}

\section{Convex geometry}\label{sec-convgeom}

\subsection{General notation for cones and Looijenga's results.}\label{ssec-notation}
In this subsection, we let $V_{\ZZ}$ be a free $\ZZ$-module of finite rank, and $V\cnec V_{\ZZ}\otimes \RR$. A {\it convex cone} in $V$ is a subset of $V$ that is invariant by multiplication by positive scalars and by  taking sums. For any subset $S$ of $V$, the convex cone \textit{generated by} $S$ is defined as the smallest convex cone containing $S$.
A convex cone $C\subset V$ is called:
    \begin{itemize}
        \item \textit{strictly convex} if its closure $\ol{C}$ contains no line;
        \item \textit{polyhedral}, respectively \textit{rational polyhedral} if $C$ contains the origin, and is  generated by finitely many elements of $V$, respectively of $V_{\ZZ}$.
    \end{itemize}
For any convex cone $C$ in $V$, we define $C^+$ as the convex cone generated by $\ol{C}\cap V_{\ZZ}$. Note that any inclusion of convex cones $C_1\subset C_2$ in $V$ is preserved by this operator, i.e., $C_1^+\subset C_2^+$. 

We define the {\it relative interior} ${\rm ri}(C)$ to be the interior of $C$ in the $\RR$-linear subspace  of $V$ generated by $C$. If $C$ has nonempty interior, then we say that $C$ has {\it full dimension} and write $C^{\circ}$ for the (relative) interior of $C$.

The following statement is contained in~\cite[Proposition-Definition 4.1]{Lo14}
\footnote{More precisely, we apply
 \cite[Proposition-Definition 4.1]{Lo14} 
 to the open cone $C^\circ$; note that $\ol{C^\circ} = \ol{C}$ since $C$ is a convex cone with nonempty interior.}

\begin{pro}[Looijenga]\label{pro_lo4.1}
 Let $C\subset V$ be a strictly convex cone with nonempty interior. Let $\Gamma$ be a subgroup of $\GL(V_{\ZZ})$ preserving $C$. Let $\Pi$ be a polyhedral cone contained in $C^+$ such that $C^\circ \subset \Gamma \cdot \Pi$. Then
    $\Gamma \cdot \Pi = C^+.$
\end{pro}

Since any rational polyhedral cone $\Pi\subset C$ automatically satisfies $\Pi\subset C^+$, this proposition shows that
$C^+$ is the largest subcone of $\overline{C}$ which may be covered by 
the $\Gamma$-translates of a rational polyhedral cone. 
This motivates the following definition.

\begin{Def}
Let $C\subset V$ be a strictly convex cone with nonempty interior. Let $\Gamma$ be a subgroup of $\GL(V_{\ZZ})$ preserving $C$. We say that the action $\Gamma \acts C$ is {\it of polyhedral type}, if there exists a polyhedral cone $\Pi \subset C^+$
such that 
$C^\circ \subset \Gamma \cdot \Pi.$   
\end{Def}

This definition is slightly different from \cite[Proposition-Definition 4.1]{Lo14}, as we do not require the cone $C$ to be open. In fact, we allow some flexibility regarding the boundary of the cone $C$ here, which is convenient in later proofs.

We say that an action $\Gamma \acts C^+$ as above has a {\it rational polyhedral fundamental domain} if there exists a rational polyhedral cone 
$\Pi \subset C^+$ such that 
    $$\Gamma \cdot \Pi = C^+,$$ 
and such that for every $\gamma \in \Gamma$, the non-emptiness $\gamma{\Pi}^\circ \cap {\Pi}^\circ \neq\emptyset$ implies $\gamma=\id$. 
This property is {\it a priori} 
stronger than the fact that $\Gamma \acts C$ is of polyhedral type; they are in fact equivalent by the foundational work~\cite{Lo14} of Looijenga.

\begin{proposition}[Looijenga]\label{pro-looij}
Let $C \subset V$ be a strictly convex cone with nonempty interior.
Let $\Gamma$ be a subgroup of $\GL(V_{\ZZ})$
preserving $C$.
The following statements are equivalent.
	\begin{enumerate}   
    	\item $\Gamma \acts C$ is of polyhedral type. 
        \item $\Gamma \acts C^+$ has a rational polyhedral fundamental domain.
	\end{enumerate}
\end{proposition} 

\begin{proof} Proposition \ref{pro-looij} is essentially \cite[Proposition 4.1, Application 4.14, and Corollary 4.15]{Lo14}; see also \cite[Lemma 3.5]{LZ22} for more details.
\end{proof}

\ssec{A descent property of actions of polyhedral type}
We define a face of a convex cone as follows, 
which is equivalent to the definition in~\cite[p. 162]{Ro70}.

\begin{definition}
   Let $C$ be a convex cone. A {\it face} of $C$ is a convex cone $F\subset C$ such that for any closed line segment $I \subset C$
with $I \cap F \ne \emptyset$, it holds either
$$I \subset  F \ \ \text{ or } 
\ \ I \cap F = \Set{ \text{one end point of } I}.$$
\end{definition}

\begin{remark}\label{rem-face-plus}
    Let $C\subset V$ be a convex cone, and let $F$ be a face of the convex cone $C^+$. Then $F$ coincides with the convex cone generated by $F\cap V_{\ZZ}$. Indeed, any element $f\in F\subset C^+$ can be written as a sum
    $f=\sum_{i=1}^n \lambda_i c_i$ with $c_i\in\ol{C}\cap V_{\ZZ}$ and $\lambda_i > 0$. Since $F$ is a face of $C^+$ and $c_i\in C^+$ for all $i$, this implies that $c_i\in F$ for all $i$, as wished.
\end{remark}

The property of being of polyhedral type descends well from a convex cone to its faces. This is the content of the following proposition. 
We will apply it to prove a descent result for the nef cone conjecture
(see Proposition~\ref{prop_descent}). For a group action $\Gamma\acts C$ and a subset $F\subset C$, we denote by 
$$\Stab(\Gamma, F) \cnec \Set{\gamma \in \Gamma :  \gamma(F) = (F)}$$ 
the {\it stabilizer} of $F$ in $\Gamma$.

\begin{pro}\label{pro_descent_face}
Let $\Gamma \acts C$ be an action of polyhedral type.
Then for each face $F$ of $C^+$, the action $\Stab(\Gamma, F) \acts F$ 
is of polyhedral type as well. 
\end{pro}

\begin{proof} 

Applying Proposition~\ref{pro_lo4.1} to the action $\Gamma\acts C$ of polyhedral type, we obtain a polyhedral cone $\Pi \subset C^+$ such that $\Gamma \cdot \Pi = C^{+}$. Let $\{F_i\}_{i\in I}$ be the relative interiors of the finitely many non-zero faces of the polyhedral cone $\Pi$. 

For each index $i \in I$, we consider the set 
$$\fF_i \cnec \Set{g \in \Gamma : gF_i \cap {\rm ri}(F) \ne \emptyset }.$$ 
We have
\begin{equation}\label{eqn-deffFi}
\fF_i = \Set{g \in \Gamma : gF_i \subset {\rm ri}(F) }.
\end{equation} 
The reverse inclusion is clear. For the direct inclusion, note that $gF_i \cap {\rm ri}(F) \ne \emptyset$ implies $gF_i \subset F$ by ~\cite[Theorem 18.1]{Ro70}, and thus, $gF_i \subset \mathrm{ri}(F)$ by~\cite[Corollary 6.5.2]{Ro70}. The set $\fF_i$ is endowed with an action of $\Stab(\Gamma,F)$ by left-multiplication. 
    
Note that for every $g,h \in \fF_i$, we have
$$\emptyset\neq F_i \subset {\rm ri}(g^{-1}F) \cap {\rm ri}(h^{-1}F),$$
and ${\rm ri}(h^{-1}F)$ and ${\rm ri}(g^{-1}F)$ both are relative interiors of faces of $C^+$. This implies that ${\rm ri}(h^{-1}F) = {\rm ri}(g^{-1}F)$, because relative interiors of distinct faces of $C^+$ are disjoint~\cite[Theorem 18.2]{Ro70}. Thus, the action of $\Stab(\Gamma,F)$ on $\fF_i$ is transitive. 
	
For each index $i \in I$, we choose one element $g_i \in \fF_i$, and let $\Phi_i \cnec g_iF_i \subset {\rm ri}(F)$. We have
$${\rm ri}(F) = \bigcup_{g \in \Gamma} ({\rm ri}(F) \cap g\Pi) = \bigcup_{i\in I} \bigcup_{g \in \fF_i} g F_i = \bigcup_{i\in I} \Stab(\Gamma,F)\cdot \Phi_i,$$
where the second equality follows from ~\eqref{eqn-deffFi}, and the third equality from the transitivity of the action of $\Stab(\Gamma,F)$ on $\fF_i$.

We introduce the following convex cone
$$\Pi_F \cnec \sum_{i \in I} \ol{\Phi_i}.$$
As a finite sum of polyhedral cones, it is a polyhedral cone. By \cite[Theorem 18.1]{Ro70}, it is contained in the face $F$. It is thus contained in $F^+$ by Remark \ref{rem-face-plus}. 
Moreover, we have
\[ {\rm ri}(F)  \ \ = \ \   \Stab(\Gamma,F)\cdot \bigcup_{i\in I} \Phi_i \ \ \ \subset \ \ \Stab(\Gamma,F)\cdot\sum_{i\in I} \Phi_i \ \ \ \subset \ \ \Stab(\Gamma,F)\cdot \Pi_F. \qedhere
\]
\end{proof}

\subsection{Assembling actions of polyhedral type in a chamber decomposition}\label{ssec-assemble-convex}

  We keep the same setting as before:
  let $C\subset V$ be a strictly convex cone with nonempty interior,
  and let $\Gamma \le \GL(V_{\ZZ})$ be a subgroup preserving $C$. 
  
  Let $\{N_Y\}_{Y \in I}$ be a collection of convex cones contained in $C$,
  with nonempty interiors, such that
  \begin{itemize}
  \item the action of $\Gamma$ on $C$ naturally induces an action by permutations on the set $\{N_Y\}_{Y \in I}$; 
  
  \item letting $M \cnec \bigcup\limits_{Y \in I} N_Y^+$, we have
  $C^{\circ}\subset M \subset C.$ 
  \end{itemize}

\begin{pro}\label{prop_glue_chamber}
 In the above setting, we assume that 
 the action of $\Gamma$ by permutations on the index set $I$ has finitely many orbits, and that for every $Y\in I$, the action $\Stab(\Gamma,N_{Y}) \acts N_Y$
is of polyhedral type.
Then, the action
  $\Gamma \acts C$ is of polyhedral type as well, and we have $C^{+} \subset C$. 
\end{pro}

\begin{proof} By assumption, we can pick
a finite set of representatives $\{ Y_1, \dots, Y_k \}$ for the orbits of $\Gamma \acts I$.
For each $Y_i$, there exists by Proposition \ref{pro-looij}
a rational polyhedral cone $\Pi_i$ inside $N_{Y_i}^+$ such that
$$\Stab(\Gamma,N_{Y_i}) \cdot \Pi_i = N_{Y_i}^+.$$
Let $\Pi$ be the convex cone generated by all of the $\Pi_{i}$, for $1\le i\le k$.
It is a rational polyhedral cone contained in $C$ by assumption, and in $C^+$ since it is rational polyhedral.
For any $N_Y$, there are an element $g \in \Gamma$ 
and an index $i \in \{1, \dots, k\}$ such that 
$$N_Y^+ = g \cdot N_{Y_i}^+ = g \cdot \Stab(\Gamma,N_{Y_i})\cdot \Pi_i \subset \Gamma \cdot \Pi,$$
thus taking the union over $I$, we obtain
$C^\circ \subset M \subset \Gamma \cdot \Pi.$
The action $\Gamma \acts C$
is thus of polyhedral type, and we can apply Proposition~\ref{pro_lo4.1} to show that
\[ C^{+} = \Gamma\cdot \Pi \subset C. \qedhere
\]
\end{proof}

    \section{Chamber decompositions of cones of divisors}\label{sec-chamber-decs}

    \subsection{Mori chambers}\label{subsec-morichamber}
    Let $X$ be a normal $\QQ$-factorial projective variety. For a $\QQ$-factorial birational contraction $(Y,f)$ of $X$, we define the {\it $f$-Mori chamber} $\Eff(X;f\colon X\dashrightarrow Y)$, or for short $\Eff(X;f)$, as the following cone
    
    $$\Eff(X;f) \cnec f^*\Nefe(Y)+\sum_{E\in \Exc(f)}\RR_{\ge 0}[E] \quad \subset \quad \Eff(X),$$
    
   \noindent where $\Exc(f)$ is the (finite)
   set of prime exceptional divisors of $f$.
    Note that it is a strictly convex cone of full dimension.

\begin{remark}\label{rem_geocone} Note that the following facts hold. 
\begin{itemize}
    \item If $\alpha\colon X\dto Y$ is a small $\QQ$-factorial modification and $\mu\colon Y \dto Z$ is a $\QQ$-factorial birational contraction, then $\alpha^*\Eff(Y;\mu) = \Eff(X;\mu \circ \alpha)$.
    \item If $\alpha\colon X\dashrightarrow Y$ is a small $\QQ$-factorial modification, we have $\Eff(X;\alpha)=\alpha^*\Nefe(Y)$. 
\end{itemize}
\end{remark}

   The following lemma motivates the choice of the name {\it Mori chambers}. It generalizes \cite[Lemma 1.5]{Ka97}, which describes pullbacks of nef cones of minimal models of a given Calabi--Yau variety $X$ as {\it chambers} within the movable cone $\Mov(X)$.

  \begin{lemma}\label{lem_generalized_ka1.5} Let $X$ be a normal $\QQ$-factorial projective variety. Let $(Y_1,f_1)$ and $(Y_2,f_2)$ be marked $\QQ$-factorial birational contractions of $X$. Then, the following are equivalent:
    \begin{itemize}
    		\item[$(1)$] $(Y_1, f_1)$ and $(Y_2, f_2)$ are isomorphic; 
    		\item[$(2)$] the two cones $\Eff(X; f_1)$ and $\Eff(X; f_2)$ coincide inside $N^1(X)_{\RR}$;
                \item[$(2')$] the two cones $f_1^{\ast}\Nefe(Y_1)$ and $f_2^{\ast}\Nefe(Y_2)$ coincide inside $N^1(X)_{\RR}$;
    		\item[$(3)$] $\Eff^\circ(X; f_1) \cap \Eff^\circ(X;f_2) \neq \varnothing$ in $N^1(X)_{\RR}$, where $\Eff^\circ(X; f_i)$ denotes the interior of $\Eff(X; f_i)$; 
                \item[$(3')$] $\mathrm{ri}(f_1^{\ast}\Nefe(Y_1)) \cap \mathrm{ri}(f_2^{\ast}\Nefe(Y_2)) \neq \varnothing$ in $N^1(X)_{\RR}$. 
    \end{itemize} 
    \end{lemma}

To prove this lemma, we first recall the following result, stated in \cite[Lemma 1.7]{HK00}. It is an application of the negativity lemma and the rigidity lemma.

   \begin{lemma}[{\cite[Lemma 1.7]{HK00}}, {\cite[Lemma 6]{CL13}}]\label{lem_hk1.7} Let $f_i: X \dashrightarrow Y_i$ be two birational contractions of normal $\QQ$-factorial projective varieties. 
    Suppose that we have a numerical equivalence 
    $$f_1^{\ast}D_1 + E_1 \equiv f_2^{\ast}D_2 + E_2 
    $$ 
    of $\RR$-divisors with $D_1$ ample on $Y_1$, $D_2$ nef on $Y_2$, 
    and $E_i$ effective $f_i$-exceptional. 
    Then $f_1 \circ f_2^{-1}\colon Y_2 \to Y_1$ is a birational morphism. 
    \end{lemma} 

   We now prove Lemma \ref{lem_generalized_ka1.5}.
   
 \begin{proof}[Proof of Lemma \ref{lem_generalized_ka1.5}]  The implications $(1)\Rightarrow (2)\Rightarrow (3)$ and $(1)\Rightarrow (2')\Rightarrow (3')$ are clear. Let us prove that $(3) \Rightarrow (1)$ and $(3') \Rightarrow (1)$. 
    
    By \cite[Theorem 6.6]{Ro70}, we have $\mathrm{ri}(f_i^{\ast}\Nefe(Y_i)) = f_i^{\ast}\Amp(Y_i)$. From  \cite[Corollary 6.6.2]{Ro70}, it follows 
    \[ \Eff^{\circ}(X; f_i) = \mathrm{ri}(\Eff(X; f_i)) = f_i^{\ast}\Amp(Y_i) + \mathrm{ri}\bigg(\sum_{E\in \Exc(f)}\RR_{\ge 0}[E]\bigg) 
    \]
    where the first equality holds because the convex cone $\Eff(X; f_i)$ has full dimension. 
    
    Assume either $(3)$ or $(3')$. Then we can take for both $i=1,2$, some $\RR$-divisor $D_i$ ample on $Y_i$ and (possibly zero) $E_i$ $f_i$-exceptional on $X$ such that 
    $f_1^{\ast}D_1 + E_1 \equiv f_2^{\ast}D_2 + E_2$. 
    Applying Lemma \ref{lem_hk1.7} twice, symmetrically, we obtain that $f_1\circ f_2^{-1}\colon Y_1 \to Y_2$ is an isomorphism. This shows that $(1)$ holds.  
    \end{proof}

\begin{lemma}\label{lem-MMMC}  
 Let $(X,\Delta)$ be a klt Calabi--Yau pair.
   Let $D$ be an effective $\RR$-divisor on $X$
   such that the pair $(X,\Delta+ D)$ is klt.
   Let $f\colon (X,\Delta+ D)\dashrightarrow (Z,\Gamma_Z)$
   be a minimal model.
  Then
   $$D \in \Eff(X;f).$$
\end{lemma}

   \begin{proof}
     
Denoting by $E_i$ for $1\le i\le k$ the prime exceptional divisors of $f$, we have
$$ D \equiv K_X+\Delta+ D 
\equiv f^*(K_Z+\Gamma_Z)+\sum_{1\le i\le k} a_iE_i,$$
where we recall that $K_Z+\Gamma_Z$ is nef and $a_i=a(E_i;Z,\Gamma_Z)-a(E_i;X,\Delta + D)\ge 0$. 
As $f_*( D) \equiv K_Z+\Gamma_Z$ and
$f_*\Eff(X)\subset \Eff(Z)$, the class of $K_Z+\Gamma_Z$ is in $\Nefe(Z)$.  Hence $D \in \Eff(X;f)$.    
   \end{proof}

\subsection{Applications of the Shokurov polytope} We first recall the following result by Shokurov and Birkar (see \cite[Section 4.7]{Fu17} for more details). 

\begin{proposition}[{\cite[Proposition 3.2(3)]{Bi11}}]\label{prop_bi3.2} Let $X$ be a normal $\QQ$-factorial projective variety. Let $D_1, \dots, D_k$ be prime divisors on $X$, and let $V$ be the vector space of $\RR$-divisors spanned by the $D_i$, for $1\le i\le k$. Then the set \[
\mathcal{N}(V) := \left\{ B \in V : (X, B) \mathrm{\ is \ log \ canonical \ and} \ K_X + B \mathrm{\ is \ nef } \right \}\] is a rational polytope.
\end{proposition}

We present two propositions whose proofs rely on this important result. 
Note that for a klt Calabi--Yau pair $(X,\Delta)$, the boundary $\Delta$ is only an $\RR$-divisor.

\begin{proposition}\label{prop_fg3.1} Let $(X, \Delta)$ be a klt Calabi--Yau pair. Then there is an effective $\QQ$-divisor $\Delta'$ such that $K_X + \Delta' \sim_{\QQ} 0$, and $(X, \Delta')$ is klt. 
\end{proposition} 

A proof of Proposition \ref{prop_fg3.1} is contained in \cite[Proof of Lemma 6.1 for $\mathbb{K}=\mathbb{R}$]{Go11}.

\begin{proposition}\label{prop_lz5.1} Let $(X, \Delta)$ be a klt Calabi--Yau pair. Then the inclusion $\Nefe(X) \subset \Nefp(X)$ holds. 
\end{proposition}

Proposition \ref{prop_lz5.1} is proved by \cite[Theorem 2.15]{LOP20} for klt Calabi--Yau varieties. For pairs, the argument is essentially the same; see \cite[Lemma 5.1(1)]{LZ22}. We present an alternative proof.

    \begin{proof}[Proof of Proposition~\ref{prop_lz5.1}]

    Let $D$ be an $\RR$-divisor whose numerical class is in $\Nefe(X)$. Take $0< \varepsilon \ll 1$ such that $(X, \Delta + \varepsilon D)$ is klt. Note that $K_X + \Delta + \varepsilon D\equiv \varepsilon D$ is nef. Let $V \subset \mathrm{Div}_{\RR}(X)$ be the vector space spanned by the components of $\Delta + \varepsilon D$. Then $\Delta + \varepsilon D \in \mathcal{N}(V)$. By Proposition \ref{prop_bi3.2}, the set $\mathcal{N}(V)$ is a rational polytope. We have 
    $$\Delta + \varepsilon D = \sum\limits_{i=1}^n r_i B_i\mbox{ for some }r_i\in\RR_{> 0}\mbox{ with }\sum\limits_{i=1}^n r_i=1,$$
    where $B_1,\ldots,B_n$ are some of vertices of $\mathcal{N}(V)$. 
    Adding $K_X$ on both sides, we get 
    $$\varepsilon D\equiv\sum\limits_{i=1}^n r_i(K_X+B_i).$$
    Thus, $\varepsilon D$ is in the convex hull of a finite set of nef $\QQ$-divisors. 
    \end{proof}

    The following result is not used in any of the later proofs, but we include it for the curious reader.

    \begin{corollary}\label{cor_generalized_lz5.1} Let $(X, \Delta)$ be a klt Calabi--Yau pair, and let $f\colon (X, \Delta) \dashrightarrow (Y, \Delta_Y)$ be a $\QQ$-factorial birational contraction. Then $\Eff(X; f)\subset \Eff(X; f)^{+}$.  
\end{corollary}

\begin{proof} This is essentially a consequence of Proposition \ref{prop_lz5.1}. Let us spell it out. By definition, 
$$\Eff(X; f) =  f^*\Nefe(Y)+\sum_{E\in\mathrm{Exc}(f)} \RR_{\ge 0} [E].$$
Since every $E$ is $\QQ$-Cartier, it suffices to show that the cone $f^*\Nefe(Y)$ is contained in $\Eff(X;f)^+$ to conclude.

By Proposition \ref{prop_lz5.1} and by injectivity of $f^{\ast}$, we have
$$f^{\ast}\Nefe(Y)\subset {\rm Conv}_{\RR}\left(f^{\ast}\Nef(Y)\cap f^{\ast}N^1(Y)_{\QQ}\right),$$ 
where ${\rm Conv}_{\RR}(\cdot)$ denotes the convex hull. 
We clearly have the inclusions $f^{\ast}\Nef(Y)\subset\ol{\Eff}(X;f)$, and $ f^{\ast}N^1(Y)_{\QQ}\subset N^1(X)_{\QQ}$. Taking intersections and then convex hulls, we obtain the inclusion wished. 
\end{proof}

We conclude this subsection by mentioning the following result, which can be of interest to the reader. We will not use it below. The proof is again an application of Proposition \ref{prop_bi3.2}.

\begin{proposition}[{\cite[Theorem 2.8]{LZ22}}]\label{prop_lz2.7}  
Let $(X, \Delta)$ be a klt Calabi--Yau pair. 
Let $\Pi \subset \Eff(X)$ be a polyhedral cone. 
Then $\Pi \cap \Nef(X)$ is a polyhedral cone as well.
Moreover, if the polyhedral cone $\Pi$ is rational, 
then $\Pi \cap \Nef(X)$ is also rational.
\end{proposition}

\begin{proof} This is claimed in \cite[Theorem 2.8]{LZ22}. We give a proof here for the sake of completeness. Since $\Eff(X)$ is spanned by numerical classes of $\ZZ$-divisors, there exists a rational polyhedral cone $\Pi'$ such that $\Pi\subset \Pi' \subset \Eff(X)$. If we prove Proposition \ref{prop_lz2.7} for $\Pi'$, then $\Pi\cap\Nef(X)=\Pi\cap \left(\Pi'\cap\Nef(X)\right)$ is a polyhedral cone, and we obtain Proposition \ref{prop_lz2.7} for $\Pi$ as well. Hence, we assume in what follows that $\Pi$ is rational polyhedral.

By Proposition \ref{prop_lz5.1}, we can assume that $\Delta$ is a $\QQ$-divisor. Let us denote by $D_1, \dots, D_r$ the prime effective $\ZZ$-divisors whose classes span the extremal rays of the rational polyhedral cone $\Pi$. 
We introduce the cone $\Pi_{{\rm div}}$ spanned by the divisors $D_i$ in the vector space $\Div(X)_{\RR}$, and we let $V$ be the $\RR$-vector space generated by the irreducible components of $\Delta$ and by $D_1, \cdots, D_r$. Then, by Proposition \ref{prop_bi3.2},
\[ \cN:= \{ B \in V : (X, B) \mathrm{\ is \ log \ canonical \ and} \ K_X + B \mathrm{\ is \ nef } \} \]
is a rational polytope. Let us define
\[ \cM:= \{ D \in \Pi_{\mathrm{div}} : (X, \Delta + D) \mathrm{\ is \ log \ canonical \ and} \ D \mathrm{\ is \ nef} \}. \] 
Then, we clearly have a relation between $\cN$ and the tranlation of $\cM$ by the divisor $\Delta$, namely
$$\Delta + \cM = \cN \cap (\Delta + \Pi_{{\rm div}}).$$

Since $\Delta$ is a $\QQ$-divisor, and $\cN$ and $\Pi_{{\rm div}}$ are a rational polytope and a rational cone respectively, we see that $\cM$ is a rational polytope. 
Thus, the projection
\[ \cP:= \{ [D] \in \Pi : D\in\Pi_{{\rm div}},\,(X, \Delta + D) \mathrm{\ is \ log \ canonical \ and} \ D \mathrm{\ is \ nef} \} \]
is a rational polytope in $N^1(X)_{\RR}$; \textit{cf.} \cite[Theorem 19.3 and its proof]{Ro70}. The cone over $\cP$ is thus a rational polyhedral cone, which we denote by $\Cone(\cP)$. 
We claim that $\Cone(\cP) = \Pi\cap \Nef(X)$. The direct inclusion is clear. For the reverse inclusion, we take an arbitrary numerical class $[D]\in \Pi\cap \Nef(X)$ with $D\in\Pi_{{\rm div}}$. For a sufficiently small rational number $\varepsilon >0$, we have a klt pair $(X, \Delta + \varepsilon D)$, and thus $\varepsilon[D]\in \cP$, as wished.
This concludes the proof. 
\end{proof}

\subsection{Decompositions of the effective and movable cones}\label{ssec_chamberdec}

The main result of this subsection is Proposition~\ref{prop_decompeffmov}, 
which is an expanded formulation of Proposition~\ref{pro-EffCD}
about the chamber decomposition of the effective cone mentioned in the introduction.

Let us start with decomposing a $\QQ$-factorial birational contraction.

\begin{lemma}\label{lem-factorize} 
Let $(X,\Delta)$ be a klt Calabi--Yau pair, and let $f\colon X\dto Z$ be a $\QQ$-factorial birational contraction. 
Then we have a factorization
$f=\mu\circ\alpha$, where $\alpha\colon X\dto Y$ is a small $\QQ$-factorial modification, and $\mu \colon Y\to Z$ is a birational morphism. 
\end{lemma}

\begin{proof}
Let $H_Z$ be an ample divisor on $Z$, and let $1 \gg \epsilon > 0$ such that the pair $(X,\Delta+\epsilon f^*H_Z)$ is klt. Since $f^{\ast}H_Z$ is big, by \cite[Theorem 1.2]{BCHM10}, we have a minimal model $\alpha\colon(X,\Delta+\epsilon f^*H_Z)\dto (Y,\Delta_Y+\epsilon \alpha_*f^*H_Z)$. Note that since $f^*H_Z$ is a movable and big divisor, the map $\alpha$ does not contract any divisors; it thus is a small $\QQ$-factorial modification. 
Let $\mu\colon Y\dto Z$ denote the composition $f\circ \alpha^{-1}$. Note that $\mu^*H_Z = \alpha_*f^*H_Z$ is nef. By Lemma \ref{lem_hk1.7}, this shows that $\mu$ is a morphism.
\end{proof}

\begin{proposition} 
\label{prop_decompeffmov}
Let $(X,\Delta)$ be a klt Calabi--Yau pair. Then we have the following inclusions: 
$$\begin{array}{ccccc}
   \Bigc(X)  & \subset & \displaystyle\bigcup_{\substack{(Z, f)\\{\rm QBC}}} \Eff(X; f) & \subset & \Eff(X)\\[-2pt]  \cup   & & \cup & & \cup \\[5pt]
   \Movo(X) & \subset & \displaystyle\bigcup_{\substack{(Y, \alpha)\\{\rm SQM}}} \Eff(X; \alpha) & \subset & \Move(X)\\
\end{array}$$ 
where the indices $(Z,f)$ (resp. $(Y,\alpha)$) range over all isomorphism classes of 
marked $\QQ$-factorial birational contractions
(resp. marked small $\QQ$-factorial modifications) of $X$, and the cones featured in the unions have disjoint interiors.

Moreover, assuming the existence of minimal models for $X$, we obtain decompositions of the cones $\Eff(X)$ and $\Move(X)$:  $$\begin{array}{ccccc}
\displaystyle\bigcup_{\substack{(Z, f)\\{\rm QBC}}} \Eff(X; f) & = & \Eff(X)\\[-2pt]  \cup   & & \cup \\[5pt]
\displaystyle\bigcup_{\substack{(Y, \alpha)\\{\rm SQM}}} \Eff(X; \alpha) & = & \Move(X)\\
\end{array}$$
\end{proposition}

\begin{proof} 
The claim about ``disjoint interiors'' follows from Lemma \ref{lem_generalized_ka1.5}. Moreover, that the union is contained in $\Eff(X)$ (resp. $\Move(X)$) is clear. 
    
We now prove 
$$\Eff(X) \subset \bigcup_{\substack{(Z,f)\\ {\rm QBC}}}\Eff(X;f)$$ 
assuming the existence of minimal models for $X$, and 
$$\Bigc(X) \subset \bigcup_{\substack{(Z,f)\\ {\rm QBC}}}\Eff(X;f)$$
unconditionally. Let $D$ be an effective $\RR$-divisor on $X$, and let $0 < \epsilon \ll 1$ such that the pair $(X,\Delta+\epsilon D)$ is klt. By the existence of minimal models for $X$, or \cite[Theorem 1.2]{BCHM10} when $D$ is big, we have a minimal model $f\colon (X,\Delta+\epsilon D)\dashrightarrow (Z,\Gamma_Z)$. 
By Lemma~\ref{lem-MMMC}, we have $\epsilon D \in \Eff(X;f)$, hence  $D \in \Eff(X;f)$.

The inclusions  
    $$\Move(X) \subset \bigcup_{\substack{\alpha:X\dashrightarrow Y\\ \mbox{\footnotesize{SQM}}}} \alpha^*\Nefe(Y)$$ 
    assuming the existence of minimal models for $X$, and 
    $$\Movo(X) \subset \bigcup_{\substack{\alpha:X\dashrightarrow Y\\ \mbox{\footnotesize{SQM}}}} \alpha^*\Nefe(Y)$$
    unconditionally have been proven in many contexts in the literature (see \cite[Theorem 2.3]{Ka97}, \cite[Propositions 4.6 and 4.7]{Wa22}, \cite[Proposition 1.1]{SX23}). To show it, we proceed as in the case of the effective cone but remark that, starting with a divisor $D \in \Move(X)$, we do not contract any divisors by running the minimal model program. 
\end{proof}

The following result is well known and appears in various contexts in the literature; see e.g., \cite[Proposition 2.4]{Ka97}, \cite[Lemma 5.1.(2)]{LZ22} (assuming the existence of good minimal models), and \cite[Theorem 3.5]{SX23}.

   \begin{corollary}\label{cor_move_movp} Let $(X,\Delta)$ be a klt Calabi--Yau pair, and assume the existence of minimal models for $X$. Then $\Move(X) \subset \Movp(X)$.
   \end{corollary}

   \begin{proof} 
By Proposition \ref{prop_decompeffmov}, we have the decomposition
   	\[ \Move(X) = \bigcup_{\substack{(Y, \alpha)\\ {\rm SQM}}} \alpha^{\ast} \Nefe(Y),  
   	\]
   	where the union is taken over all the small $\QQ$-factorial modifications $(Y, \alpha)$ of $X$. For each $(Y,\alpha)$, we have $\Nefe(Y) \subset \Nefp(Y)$ by Proposition \ref{prop_lz5.1}.  
    Moreover, since $\alpha^*\Nefe(Y) \subset \Mov(X)$ and since the operator $\phantom{a}^+$
    preserves inclusions of convex cones, we obtain $\alpha^*\Nefe(Y)\subset \Movp(X)$. This concludes.
    \end{proof}

    \subsection{ Induced chamber decompositions of polyhedral cones} 

In this subsection, we introduce a chamber decomposition for a polyhedral cone in the effective cone, which is naturally induced by the Mori chamber decompositions described in Proposition \ref{prop_decompeffmov}.

\begin{proposition}\label{prop_finitepi}
Let $(X,\gD)$ be a klt Calabi--Yau pair with $\gD$ a $\QQ$-divisor satisfying Assumption \ref{hyp-GMMPX}. Let $\Pi\subset\Eff(X)$ be a polyhedral cone.  
\begin{enumerate}
\item 
There are only finitely many $\QQ$-factorial birational contractions $f_i\colon X\dto Y_i$ (for $1\le i\le r$) such that the interior of the Mori chamber $\Eff(X;f_i)$ intersects $\Pi$. Moreover,
$$\Pi = \bigcup_{i = 1}^r \Pi\cap \Eff(X;f_i).$$
    
    \item For each $\Eff(X;f_i)$ as in (1),
    the intersection $\Pi\cap\overline{\Eff(X;f_i)}$ is polyhedral, and even rational polyhedral if $\Pi$ is rational polyhedral itself. 
\end{enumerate}
\end{proposition}

Before proving that proposition, we state a result due to Shokurov \cite{Sh96} and Kawamata \cite{Ka11}. The precise statement can be found in the book \cite{KawamataBook}.

\begin{theorem}[{\cite[Theorems 2.10.3, 2.10.4, Corollary 2.10.6]{KawamataBook}}]\label{thm-kawadec}
    Let $X$ be a normal projective $\QQ$-factorial variety. Let $E_1,\ldots,E_m$ be effective $\QQ$-divisors on $X$ such that $(X,E_i)$ is a klt pair for all $i$. Let $V$ be the linear span of the $E_i$ in $Z^1(X)_{\RR}$ and consider the convex hull
    $$P' = {\rm Conv}_{\RR}(E_1,\ldots,E_m) \subset V, \text{ and }
    $$
    $$P=\{D\in P' : 
    [K_X+ D]\in \overline{\Eff}(X)\}.$$
   Assume the following two conditions hold for each $\QQ$-divisor $D\in P$:
    \begin{enumerate}
        \item[(a)] The pair $(X,D)$ has a good minimal model $\alpha\colon(X,D)\dashrightarrow (Y,D_Y)$. We denote by $g\colon(Y,D_Y)\to Z$ the canonical model induced by semiampleness of $K_Y+D_Y$. (Both $\alpha$ and $g$ may depend on $D$.)
    \item[(b)] There is a rational polytope $P'_D\subset V$ with $D$ in its interior such that, for any $\QQ$-divisor $D'$ in
    $$P_D = \{D'\in P'_D\cap P' : [K_Y+\alpha_*D']\in\overline{\Eff}(Y/Z)\},$$
    the morphism $g:(Y,\alpha_*D')\to Z$ admits a good minimal model.
    \end{enumerate}
For any birational map $\alpha\colon X\dashrightarrow Y$, we set
    $$Q_{\alpha}\cnec \{B\in P : \alpha\mbox{ is a minimal model for }(X,B)\}.$$
    Then there exists a finite decomposition into disjoint subsets
    $$P=\bigsqcup_{i=1}^r Q_{i}$$
 such that:
    \begin{enumerate}
        \item For any birational map $\alpha\colon X\dashrightarrow Y$ such that $Q_{\alpha}$ is non-empty, we can write
$$Q_{\alpha} = \bigsqcup_{i\in I(\alpha)} Q_i,$$
where $I(\alpha)$ is a subset of $\{1,\ldots,r\}$.
        \item Each closure $\overline{Q_{i}}$ is a rational polytope.
    \end{enumerate}
    
\end{theorem}

\begin{remark}
Theorem~\ref{thm-kawadec} is a simplification
of the combination of~\cite[Theorems 2.10.3, 2.10.4, Corollary 2.10.6]{KawamataBook}.
In Theorem~\ref{thm-kawadec}, 
the chambers $Q_i$ correspond to $Q_{j,k}$ in~\cite[Theorem 2.10.4]{KawamataBook}. The cones $P_j$ in~\cite[Theorem 2.10.3]{KawamataBook} do not appear in Theorem~\ref{thm-kawadec}.
\end{remark}

\begin{remark}
 Note that \cite[Theorems 2.10.3, 2.10.4]{KawamataBook} originally require Assumptions (a) and (b) to hold for each $D\in P$, not just for the $\QQ$-divisors. However, Kawamata's proofs work well over $\QQ$. See also \cite[Corollary 2.10.6 and its proof]{KawamataBook}. The argument may only possibly involve $\RR$-divisors that are not $\QQ$-divisors when applying the induction step (in \cite[Proof of Theorem 2.10.3, Page 151, Last Paragraph]{KawamataBook}): If we fix a $\QQ$-divisor $B$ and apply Kawamata's argument, we only need the boundary $\partial(P'_B\cap P')$ to be a union of rational polytopes to continue the induction argument entirely over $\QQ$. For that, it suffices that $P'_B$ and $P'$ be rational polytopes. We enforce that in the statement of Theorem \ref{thm-kawadec}.
\end{remark}

\begin{proof}[Proof of Proposition \ref{prop_finitepi}]
We start with a remark: If Proposition \ref{prop_finitepi} holds for a given polyhedral cone $\Pi'\subset\Eff(X)$, then it also holds for any polyhedral subcone $\Pi\subset \Pi'$. Indeed, the finiteness of the chamber decomposition for $\Pi$ follows from that for $\Pi'$, while the polyhedrality for $\Pi$ comes from the fact that
$$\Pi\cap\overline{\Eff(X;f_i)}=\Pi\cap(\Pi'\cap\overline{\Eff(X;f_i)}),$$
    and the intersection of two polyhedral cones is polyhedral.
Using this remark, we are reduced to the following essential case: In what follows, we assume that $\Pi \subset \Eff(X)$ is a full-dimensional rational polyhedral cone.

Let $B_1,\ldots, B_m\in Z^1(X)_{\QQ}$ be effective $\QQ$-divisors whose numerical classes span the cone $\Pi$. Up to replacing them by small enough multiples, we can ensure that the pair $(X,\gD + B)$ is klt
    for any $\RR$-divisor $B\in{\rm Conv}_{\RR}(B_1,\ldots, B_m)$.

 We want to apply Theorem \ref{thm-kawadec} to $X$ and  $E_i\cnec\Delta+ B_i$.
Note that $P = P'$ in this case.
 Assumption \ref{hyp-GMMPX}(a) clearly implies Condition (a) in Theorem \ref{thm-kawadec}. We now check Condition (b). Fix $D\in P$.
Since $(Y,D_Y)$ is klt, we can pick a small enough polytope $P'_D$ around $D$ such that, for any $D'\in P_D = P'_D \cap P'$, the pair $(Y,\alpha_*D')$ is klt as well. Assumption \ref{hyp-GMMPX}(b), together with \cite[Theorem 2.12]{HX13} and \cite[Theorem 2]{KawamataBook}, shows that if $D'$ is a $\QQ$-divisor in $P_D$, the pair $(Y,\alpha_*D')$ indeed admits a relative good minimal model over $Z$.

 We decompose $P\cnec \Delta+{\rm Conv}_{\RR}(B_1,\ldots, B_m)$ according to Theorem \ref{thm-kawadec}:
\begin{equation}\label{eq-cover}
\Delta+{\rm Conv}_{\RR}(B_1,\ldots, B_m) = \bigsqcup_{i=1}^r Q_i.
\end{equation} 
For each $i$, by Assumption (a) we choose a birational map $f_i\colon X\dashrightarrow Y_i$ that is a minimal model of at least one pair of the form $(X,\Delta+B)$ with $\Delta+B\in Q_i$. By Theorem \ref{thm-kawadec}(1), we have $Q_i\subset Q_{f_i}$. 

Let us denote by $[\cdot]$ the projection from $Z^1(X)_{\RR}$ to $N^1(X)_{\RR}$. For each $i$, we claim that the inclusions
\begin{equation}\label{eq-inclQi}
\RR_{> 0}\cdot [Q_{i}-\Delta]\; \subset \;
\RR_{> 0}\cdot [Q_{f_i}-\Delta]
\;\subset\; \Pi\cap\Eff(X;f_i)
\end{equation}
hold. The first inclusion is clear. For the second one, let $\Delta+B\in Q_{f_i} \subset P$. Clearly, $[B]\in\Pi$ holds. We set $\Gamma\cnec {f_i}_*(\Delta+B)$. Recall that $Q_i\subset Q_{f_i}$, hence the birational map $f_i$ is a minimal model for $(X,\Delta+B)$. Then $[B]\in \Eff(X;f_i)$ by Lemma~\ref{lem-MMMC}.

From Equations \eqref{eq-cover} and \eqref{eq-inclQi} for all $i$,
we derive that 
\begin{equation}\label{eqn-unionPi}
\Pi = \{0\} \cup \bigcup_{i=1}^r \RR_{>0}\cdot [Q_i-\Delta] \subset \bigcup_{i=1}^r \Pi\cap \Eff(X;f_i) \subset \Pi.
\end{equation}
This implies the equality 
$$\Pi = \bigcup_{i = 1}^r \Pi\cap \Eff(X;f_i).$$
If the interior of a Mori chamber $\Eff(X,g)$ intersects $\Pi$, it must intersect the interior of a chamber $\Pi\cap\ol{\Eff(X;f_i)}$, and in particular ${\Eff}^{\circ}(X;f_i)$. By Lemma \ref{lem_generalized_ka1.5}, this only happens if $\Eff(X,g) = {\Eff}(X;f_i)$, which proves (1).

We now show that for all $1\le i\le r$ such that the intersection
$\Pi\cap \Eff^\circ(X;f_i)$ is non-empty,
the intersection $\Pi\cap\overline{\Eff(X;f_i)}$ is  rational polyhedral. By Theorem \ref{thm-kawadec}(1) and (2), 
and since $\Pi\cap\ol{\Eff(X;f_i)}$ is convex, 
it suffices to prove the following equality of closed cones:
$$\Pi\cap\ol{\Eff(X;f_i)}=\RR_{>0}\cdot [\overline{Q_{f_i}} -\Delta].$$
The reverse inclusion is shown in \eqref{eq-inclQi}. Arguing by contradiction, we assume that the direct inclusion fails. Then there exists an open ball 
$U$  in $\Pi\cap\overline{\Eff(X;f_i)}$ that is disjoint from $\RR_{>0}\cdot [\overline{Q_{f_i}}-\Delta]$. Equation \eqref{eq-cover} and Theorem \ref{thm-kawadec}(1) show that up to shrinking $U$, we have $U\subset \RR_{>0}\cdot [Q_j-\Delta]$ 
 for some $1\le j\le r$ with $Q_j\not\subset \overline{Q_{f_i}}$. 
 Consequently, by~\eqref{eq-inclQi} for $Q_j$, the open ball $U$ is contained in $\Pi\cap \Eff(X;f_j)$, 
 which then implies
 $\Eff^\circ(X;f_i) \cap \Eff^\circ(X;f_j) \ne \emptyset$.
 It follows from Lemma \ref{lem_generalized_ka1.5} that $f_i$ decomposes as $h\circ f_j$ for an isomorphism $h\colon Y_j\to Y_i$, thus $Q_{f_i}=Q_{f_j}$. But $Q_j$ is contained in $Q_{f_j}$, which is a contradiction.
\end{proof}

\section{A descent result for the nef cone conjecture} \label{sec-descent}

\subsection{Stabilizers of Mori chambers}\label{ssec-stab-mori}
    
   Let $f\colon  (X, \Delta) \dashrightarrow (Y, \Delta_Y)$ be 
   a $\QQ$-factorial birational contraction. 
   Define 
   $$ \PsAut(X,\Delta;f) \cnec 
   \PsAut(X,\gD) \cap (f^{-1} \circ \Aut(Y) \circ f), 
   $$
    $$\Aut(X,\Delta;f) \cnec \Aut(X,\gD)  \cap (f^{-1} \circ \Aut(Y) \circ f).$$

   \begin{corollary}\label{cor_equality_groups} 

       Let $f\colon  (X, \Delta) \dashrightarrow (Y, \Delta_Y)$ 
       be a $\QQ$-factorial birational contraction. Then 
         \begin{equation*}
             \begin{split}
        \PsAut(X, \Delta; f) &  = \Stab(\PsAut(X, \Delta) ,  \,  \Eff(X;f))
        \\
        &= \Stab(\PsAut(X, \Delta) ,  \,  f^{\ast}\Nef(Y)),          
             \end{split}
         \end{equation*} 
         with respect to the action $\PsAut(X, \Delta) \acts N^1(X)_\RR$.
   \end{corollary}

    \begin{proof}
Let $g \in \PsAut(X, \Delta)$.
The pullback $g^*$ stabilizes $\Eff(X;f)$ if and only if
$$ \Eff(X;f)= g^*\Eff(X;f) = \Eff(X;fg),$$
where the last equality follows from Remark~\ref{rem_geocone}.
By Lemma~\ref{lem_generalized_ka1.5},
this is equivalent to the property that
the marked $\QQ$-factorial birational contractions
    $(Y,fg)$ and $(Y,f)$ are isomorphic;
    namely $\gamma \cnec fgf^{-1} \in \Aut(Y)$. 
    This proves the first equality. 
Again, by Lemma~\ref{lem_generalized_ka1.5},
that $(Y,fg)$ and $(Y,f)$ are isomorphic is equivalent to
\begin{equation}\label{eqn-fgNefeY}
    f^*\Nefe(Y) = (fg)^*\Nefe(Y) = g^*f^*\Nefe(Y),
\end{equation}
where we use Lemma~\ref{lem_funpull} for the second equality.
Since $\Nefe(Y)$ is dense in $\Nef(Y)$,~\eqref{eqn-fgNefeY}
is equivalent to $f^*\Nef(Y) = g^*f^*\Nef(Y)$.
This proves the second equality.
\end{proof}

\begin{lemma}\label{lem_autonexceptional}  Let $f\colon  (X, \Delta) \dashrightarrow (Y, \Delta_Y)$ 
       be a $\QQ$-factorial birational contraction. Then the action of $\Aut(X,\Delta;f)$ preserves the cone $\Eff(X)\cap\ker f_*$ in $N^1(X)_{\RR}$.
\end{lemma}

\begin{proof}
     It suffices to check that $\Aut(X,\Delta;f)$ preserves the linear subspace $\ker f_*$. Let $g\in\Aut(X,\Delta;f)$. We can write 
$f\circ g = h\circ f,$ where $h$ is an automorphism of $Y$. It follows that $$\ker f_*\subset \ker (h\circ f)_*  = \ker (f\circ g)_* = (g^{-1})_*(\ker f_*),$$
which have the same dimension, thus coincide. 
\end{proof}

    \subsection{Descending the nef cone conjecture}

The following statement is a descent result for the nef cone conjecture, similar to \cite[Lemma 3.4]{To10} for surface pairs.

\begin{proposition}\label{prop_descent}
    Let $(X,\Delta)$ be a klt Calabi--Yau pair.
    Assume that the nef cone conjecture holds for $(X,\Delta)$. 
    Then for any birational morphism $f\colon (X,\Delta)\to (Y,\Delta_Y)$ with $Y$ being $\QQ$-factorial, the action 
    $$\Aut(X,\Delta;f) \acts f^*\Nef(Y)$$
is of polyhedral type, and $\Nefe(Y)=\Nefp(Y)$. 
    In particular, the nef cone conjecture holds for $(Y, \Delta_Y)$. 
\end{proposition}

\begin{proof}

Intersecting with $\Aut(X,\gD)$ on both sides of the equality in Corollary~\ref{cor_equality_groups}, we have
$$\Aut(X, \Delta; f) = \Stab(\Aut(X, \Delta),\, f^{\ast}\Nef(Y)).$$
Note that $f^*\Nefp(Y)$ is a face of the cone $\Nefp(X)$. Using the nef cone conjecture for $(X,\Delta)$ and
Proposition~\ref{pro_descent_face}, we deduce
that $\Aut(X, \Delta; f)  \acts f^{\ast}\Nefp(Y)$
is of polyhedral type. 
That action is isomorphic to $f\circ\Aut(X, \Delta; f)\circ f^{-1}  \acts \Nefp(Y)$ under $f^*$, so the latter is of polyhedral type as well. 

By our assumption on $(X,\Delta)$,
we have $\Nefp(X)=\Nefe(X)$, so
$$f^*\Nefp(Y)\subset \Nefp(X)=\Nefe(X)\subset\Eff(X).$$ 
Since $f_*f^* = \id_{N^1(Y)_{\QQ}}$ (see \S\ref{subssec-pfwd}) and $f_*\Eff(X) \subset \Eff(Y)$, we deduce that
$\Nefp(Y)\subset\Eff(Y)$.
 Together with Proposition \ref{prop_lz5.1}, this proves that $\Nefe(Y)=\Nefp(Y)$.

Note that by definition $f\circ\Aut(X, \Delta; f)\circ f^{-1}$ is a subgroup of $\Aut(Y)$. Clearly it is also a subgroup of $\Aut(Y,\Delta_Y)$.
Thus, we can apply Proposition \ref{pro-looij} to conclude that the nef cone conjecture holds for $(Y,\Delta_Y)$.
\end{proof}

\ssec{Finiteness statement from cone conjecture} 
It is well-known that the cone conjectures imply various finiteness statements. In this subsection, we recall one such result for later use. We also provide a proof for the sake of completeness.

\begin{pro}\label{prop_BCfini}
Let $(X,\Delta)$ be a klt Calabi--Yau pair.
    Assume that the nef cone conjecture holds for $(X,\Delta)$. Then, up to isomorphism of pairs, there are only finitely many pairs $(Y,\Delta_Y)$ with $Y$ being $\QQ$-factorial arising from birational morphisms  $(X,\Delta) \to (Y,\Delta_Y)$. 
\end{pro}

\begin{proof}
By assumption, $\Aut(X,\Delta) \acts \Nefp(X)$ has a rational polyhedral fundamental domain $\Pi$.
    Consider the set $S$ of isomorphism classes of pairs $(Y,\Delta_{Y})$ with $Y$ being $\QQ$-factorial admitting a birational morphism $f\colon(X,\Delta)\to (Y,\Delta_Y)$. 
    Sending $(Y,\Delta_{Y})$ to $f^*\Nef(Y)$ defines an injection from the set $S$ 
    into
    the set of $\Aut(X,\Delta)$-orbits of closed rational faces of the cone $\Nef(X)$. Each orbit contains a face of $\Pi$, hence $S$ injects into the set of faces of $\Pi$, which is finite.
\end{proof}

\section{Proof of Theorem~\ref{main_thm-equiv}}\label{sec-PfThMain}

We state and prove the following theorem, which is slightly more general, and from which Theorem~\ref{main_thm-equiv} directly follows. Item $(2')$ is first introduced by Li--Zhao in \cite[Conjecture 1.2(1)]{LZ22}.

\begin{theorem}\label{thm_equivalence_detailed}  Let $(X,\gD)$ be a klt Calabi--Yau pair. Consider the following statements. 
        \begin{enumerate}
            \item The movable cone conjecture holds for $(X,\gD)$.
            \item The effective cone conjecture holds for $(X,\gD)$. 
            \item[$(2')$] There is a polyhedral cone $\Pi\subset \Eff(X)$ satisfying 
            $$\Movo(X) \subset \PsAut(X, \Delta)\cdot \Pi.$$
            \item 
            \begin{enumerate}
            \item The nef cone conjecture holds for each klt pair $(X',\gD')$ obtained by a small $\QQ$-factorial modification from $(X,\gD)$.
            \item 
            There exist only finitely many $(X', \Delta')$ up to isomorphism of pairs arising as small $\QQ$-factorial modifications of $(X,\Delta)$. 
            \end{enumerate}
            \item 
            \begin{enumerate}
            \item The Mori chamber cone conjecture holds for each 
            $\QQ$-factorial birational contraction $f\colon (X,\gD)\dto (Y,\gD_Y)$. 
            \item 
            There exist only finitely many $(Y,\Delta_Y)$ up to isomorphism of pairs arising as $\QQ$-factorial birational contractions of $(X,\Delta)$. 
            \end{enumerate}
        \end{enumerate} 
   Then the following assertions hold.

    \begin{itemize}
    \item[(i)] We have $(3) \Leftrightarrow (4) \Rightarrow (2)$ and 
    $[(1) \text{ or } (2)] \Rightarrow (2')$. 
    
    \item[(ii)] Assuming the inclusion $\Move(X) \subset \Movp(X)$, 
    we have $(3) \Rightarrow (1)$. 
    
    \item[(iii)] Assuming that, for some $\QQ$-divisor $\gG$, the pair $(X,\gG)$ is klt Calabi--Yau and satisfies Assumption \ref{hyp-GMMPX}. We have $(2') \Rightarrow (3)$. 
    \end{itemize}
    \end{theorem}

\begin{proof}[Proof of Theorem ~\ref{main_thm-equiv} using Theorem \ref{thm_equivalence_detailed}.]

The existence of minimal models for $X$ implies the inclusion $\Move(X) \subset \Movp(X)$, by Corollary \ref{cor_move_movp}.
By Lemma \ref{lem-GMMPind}, the existence of good minimal models in dimension $\dim X$ implies Assumption \ref{hyp-GMMPX} for any klt Calabi--Yau pair of the form $(X,\Gamma)$ with $\gG$ a $\QQ$-divisor, and there exists at least one such pair on $X$ by Proposition \ref{prop_fg3.1}.
\end{proof}

    Let us note that the implication $[(1) \text{ or } (2)] \Rightarrow (2')$ is clear. In the next subsections, we prove the remaining implications.

\subsection{The equivalence between (3) and (4)}\label{subsection_equivalence_3&4}

\begin{proof}[Proof of (3) $\Leftrightarrow$ (4) in Theorem~\ref{thm_equivalence_detailed}(i)]
The implication $(4) \Rightarrow (3)$ is quite straight-forward. Let us explain that.
Since small $\QQ$-factorial modifications are
$\QQ$-factorial birational contractions,
(4b) implies (3b).
Let  $f\colon  (X,\gD) \dto (X',\gD)$ be a
small $\QQ$-factorial modification.
Then $\Eff(X;f) = f^*\Nefe(X')$ and 
$$f^{-1} \circ \Aut(X',\Delta') \circ f \subset \PsAut(X,\gD).$$
It follows that
$$\PsAut(X,\Delta;f) =
   f^{-1} \circ \Aut(X',\Delta') \circ f,$$
(see the definition in Subsection \ref{ssec-stab-mori}), thus (4a) implies (3a).

Now we assume (3).
First we prove (4b).
Using (3b), we fix $(X_i,\Delta_i)$ for $i=1,\ldots, r$ to be a set of representatives of the (finitely many) klt pairs obtained by small $\QQ$-factorial modifications from $(X,\Delta)$. For each $i$, we also fix an arbitrary marking $\alpha_i\colon (X,\Delta)\dto (X_i,\Delta_i)$. Since we assume (3a) and by Proposition \ref{prop_BCfini}, for each $1\le i\le r$, we have finitely many $(Y_{i,j},\Delta_{i,j})$, $1\le j\le s_i$, arising from birational morphisms $\mu_{i,j}\colon (X_i,\Delta_i) \to (Y_{i,j},\Delta_{i,j})$ with $Y_{i,j}$ being $\QQ$-factorial. Note that we take and fix one $\mu_{i, j}$ for each pair of $i, j$.

We fix an arbitrary $\QQ$-factorial birational contraction $f\colon (X,\Delta)\dto (Y,\Delta_Y)$. By Lemma \ref{lem-factorize}, we factorize $f = \mu\circ \alpha$, where $\alpha\colon (X,\Delta)\dto (X',\Delta')$ is a small $\QQ$-factorial modification and $\mu\colon (X',\Delta')\to (Y,\Delta_Y)$ is a birational morphism with $Y$ being $\QQ$-factorial.
By definition of the pairs $(X_i,\Delta_i)$, we can find an index $1\le i\le r$ and an isomorphism $\beta\colon (X_i,\Delta_i)\eto (X',\Delta')$. 
Since $\mu \circ \gb \colon  X_i \to Y$ 
is a birational morphism,
by definition of the $\mu_{i,j}$ we can then find an index $1\le j\le s_i$, and an isomorphism $(Y,\Delta_Y) \eto (Y_{i,j},\Delta_{i,j})$, which implies (4b).

We now prove (4a). We keep the notation of the previous paragraph, 
notably the factorization $f = \mu \circ \ga$.
We have
$$\PsAut(X,\Delta;f) = \alpha^{-1}\circ\PsAut(X',\Delta';\mu)\circ\alpha,$$
and recall from Remark~\ref{rem_geocone} that
$$\alpha^*\Eff(X';\mu)=\Eff(X;f).$$

By (3a) and Proposition \ref{prop_descent},
we have a rational polyhedral cone $\Pi$ inside $\mu^*\Nef(Y)$ satisfying
$$\Aut(X', \Delta';\mu)\cdot\Pi = \mu^*\Nefe(Y)=\mu^*\Nefp(Y).$$
This implies $\Eff(X';\mu)=\Effp(X';\mu)$
by the definition of a Mori chamber.
Define $\gS$ as the convex cone spanned by $\Pi$ and 
by the prime exceptional divisors 
$E_1, \ldots,E_\ell$ of $\mu$. 
Clearly, $\gS$ is a rational polyhedral cone contained in $\Eff(X';\mu)$. By Lemma \ref{lem_autonexceptional}, the action of $\Aut(X',\Delta';\mu)$ preserves the cone
$$\Eff(X')\cap\ker \mu_* = \sum_{k = 1}^\ell \RR_{\ge 0}[E_k].$$
We thus have
\begin{equation}
    \begin{split}
        \Aut(X',\Delta';\mu)\cdot\Sigma 
 & = \Aut(X',\Delta';\mu)\cdot\left(\Pi + \sum_{k = 1}^\ell \RR_{\ge 0}[E_k]\right) \\
& = \mu^*\Nefp(Y) + \sum_{k = 1}^\ell \RR_{\ge 0}[E_k]
=\Eff(X';\mu),
    \end{split}
\end{equation}
and the equality also holds
if the group $\Aut(X',\Delta';\mu)$
gets replaced by the larger group $\PsAut(X',\Delta';\mu)$.
Hence 
$\PsAut(X',\Delta';\mu) \acts \Eff(X';\mu)$
satisfies the cone conjecture by Proposition~\ref{pro-looij}.
\end{proof}

\subsection{Assembling the cone conjectures of chambers}

\begin{proof}[Proof of (4) $\Rightarrow$ (2) in Theorem~\ref{thm_equivalence_detailed}(i) and of Theorem~\ref{thm_equivalence_detailed}(ii)]

\hfill

Assume that (4) holds. By \S \ref{subsection_equivalence_3&4}, (3) holds as well.
We want to prove (2), respectively (1)
assuming $\Move(X) \subset \Movp(X)$. By Proposition~\ref{prop_decompeffmov}, we have the following inclusions:
\begin{align}
\Movo(X) &\subset \bigcup_{\substack{(Y, \alpha)
\\{\rm SQM}}} \Eff(X; \alpha\colon X\dto Y) \subset \Move(X),
\label{eq-dec1}\\
\mbox{and}\quad\Bigc(X) &\subset \bigcup_{\substack{(Z, f)\\{\rm QBC}}} \Eff(X; f\colon X\dto Z) \subset \Eff(X).\label{eq-dec2}
\end{align}
Since we assume (4a), we have $\Eff(X;f)=\Effp(X;f)$ for every $\QQ$-factorial birational contraction $f:X\dto Z$.
This puts us in the setting of Subsection \ref{ssec-assemble-convex}. We can check that the assumptions of Proposition~\ref{prop_glue_chamber} are satisfied: Since we are assuming (4b), there are finitely many $\QQ$-factorial birational contractions of $X$ modulo $\PsAut(X,\Delta)$. Our assumption of (4a), together with Corollary~\ref{cor_equality_groups}, ensures that each induced action 
$$\mathrm{Stab}(\PsAut(X,\Delta), \,  \Eff(X;f))\;\acts\Eff(X;f)$$
is of polyhedral type.

Hence, we can apply Proposition~\ref{prop_glue_chamber} to the unions \eqref{eq-dec1} and \eqref{eq-dec2}, both with the $\PsAut(X,\Delta)$-action. It shows that 
$$\PsAut(X,\Delta)\acts\Move(X)\mbox{ \ \ and \ \ }\PsAut(X,\Delta)\acts\Eff(X)$$ 
are of polyhedral type, and provides the inclusions
$$\Movp(X)\subset\Move(X)\mbox{ \ \ and \ \ }\Effp(X)\subset\Eff(X).$$ 
We have the reverse inclusions $\Move(X) \subset \Movp(X)$ by assumption, and $\Eff(X) \subset \Effp(X)$ unconditionally. Now we use Proposition \ref{pro-looij} to conclude.
\end{proof}

\subsection{From the effective cone to the nef cone conjectures} 

\begin{proof}[Proof of Theorem~\ref{thm_equivalence_detailed}(iii)]

\hfill

We assume $(2')$: Let $\Pi\subset\Eff(X)$ be a polyhedral cone such that 
\begin{equation}\label{eqn-inclusion61iii}
    \Movo(X)\subset\PsAut(X,\Delta)\cdot \Pi.
\end{equation}
By replacing $\Pi$ with some $\PsAut(X,\Delta)$-translate, we can assume that $\Pi$ intersects the ample cone $\Amp(X)$. 
By Proposition \ref{prop_finitepi}, 
the cone $\Pi$ only intersects the interiors of finitely many Mori chambers. 
In particular, the cone $\Pi$ only intersects the interiors of finitely many Mori chambers of the form 
$\Eff(X;\alpha\colon X\dto Y)$ with $\alpha$ a small $\QQ$-factorial modification. Let us list these chambers as 
\begin{equation}\label{list-EffSQM}
\Eff(X;\alpha_1\colon X\dto X_1),\ldots, 
\Eff(X;\alpha_r\colon X\dto X_r).    
\end{equation}

    Let us first show (3b): We will show that any klt pair $(Y,\Delta_Y)$ arising as a small $\QQ$-factorial modification of $(X,\Delta_X)$ satisfies $Y\simeq X_i$ for some $1\le i\le r$.
 We take an arbitrary small $\QQ$-factorial modification $\alpha:X\dto Y$ as a marking. By assumption, there is an element $\gamma\in \PsAut(X,\Delta)$ such that the cone $\Pi$ intersects the interior of ($\alpha\circ\gamma$)-Mori chamber:
    $$\Eff^{\circ}(X;\alpha \circ \gamma) = \gamma^*\Eff^{\circ}(X;\alpha\colon X\dto Y) = \gamma^*\alpha^*\Amp(X)\subset\Movo(X)$$
    (see also Remark \ref{rem_geocone}). In our exhaustive list \eqref{list-EffSQM}, there is an index $j$ such that $\Eff(X;\alpha \circ \gamma)=\Eff(X;\alpha_j)$. By Lemma \ref{lem_generalized_ka1.5}, this shows that $X_j$ and $Y$ are isomorphic as varieties, 
    which implies (3b).

    We now proceed to show (3a). 
    Within the list~\eqref{list-EffSQM}, let
\begin{equation}\label{list-Effpsaut}
    \Eff(X;\gamma_1\colon X\dto X), \ldots, 
\Eff(X;\gamma_s\colon X\dto X)
\end{equation}
be the chambers of the form  $\Eff(X;\gamma\colon X\dto X)$ with $\gamma\in\PsAut(X,\Delta)$.  
    By Proposition \ref{prop_finitepi}, for each $1\le j\le s$, the intersection 
    $$(\gamma_j^{-1})^{\ast}\Pi \cap \Nef(X)\subset\Eff(X)$$ 
    is a polyhedral cone. Hence, the convex cone $\Sigma$ generated by these finitely many intersections is a polyhedral cone contained in $\Nefe(X)$. 
    
    By Proposition~\ref{pro-looij},
    it now suffices to show that 
     $$\Amp(X) \subset \Aut(X, \Delta)\cdot \Sigma.$$ 
     We take $D \in \Amp(X)$. 
     By the inclusion given in \eqref{eqn-inclusion61iii}, there is an element $\beta \in \PsAut(X, \Delta)$ such that $\beta^{\ast}D \in \Pi$. 
     In particular $\beta^{\ast}\Amp(X) \cap \Pi \neq \varnothing$, so the chamber 
     $\Eff(X;\beta) = \beta^{\ast}\Nefe(X)$ coincides with one of the chambers $\Eff(X;\gamma_i)$ in our  list \eqref{list-Effpsaut} by definition.
     By Lemma \ref{lem_generalized_ka1.5}, 
     $(X, \beta)$ and $(X, \gamma_i)$ are isomorphic as marked small $\QQ$-factorial modifications of $(X,\Delta)$, i.e., $\beta\gamma_i^{-1}\in \Aut(X, \Delta)$.
We conclude that  
     $$(\beta\gamma_i^{-1})^{\ast}D = (\gamma_i^{-1})^{\ast} \beta^{\ast}D \in (\gamma_i^{-1})^{\ast}\Pi \cap \Nef(X)\subset \Sigma,$$ and thus
     $\Amp(X) \subset \Aut(X, \Delta)\cdot \Sigma$. 
\end{proof}

\subsection{Equivalence of cone conjectures in dimension two}
    
In dimension two, the work of Totaro implies the following result.  

    \begin{cor}\label{cor_eff_two} 
   Let $(X, \Delta)$ be a klt Calabi--Yau pair of dimension $2$. Then the four statements given in
    Theorem~\ref{main_thm-equiv} hold for $(X,\Delta)$.
    
    In particular,
    there exists a rational polyhedral fundamental domain 
    for the action of $\Aut(X, \Delta)$ on $\Eff(X)$. 
    \end{cor} 

    \begin{proof}
    In dimension two, the nef cone conjecture was proven by Totaro \cite[Theorem 4.1]{To10} for any klt Calabi--Yau pair $(X,\Delta)$. Since in dimension two, isomorphisms in codimension one are exactly  isomorphisms, and since the existence of good minimal models holds in dimension two, this shows that all four statements of
    Theorem~\ref{main_thm-equiv} hold, for any klt Calabi--Yau surface pair $(X,\Delta)$. 
    \end{proof}

Let us make a short comment about the effective cone conjecture for surface pairs. In dimension two, the duality between numerical classes of divisors in $N^1(X)_{\RR}$ and of curves in $N_1(X)_{\RR}$ allows to identify the cone of effective divisors $\Eff(X)$ with the cone of effective curves ${\rm NE}(X)$. Embracing that point of view, the equivalence of Statements (2) and (3) in Theorem \ref{main_thm-equiv} can be partially recovered by a general duality argument (see \cite[Proposition-Definition 4.1]{Lo14}), together with the well-known equality $\Eff(X) = \Effp(X)$. The latter equality is a consequence of the Zariski decomposition of pseudo-effective $\RR$-divisors on a surface, and of the equality $\Nefe(X)=\Nefp(X)$ (given by the nef cone conjecture).

\section{Cone conjectures for Schoen threefolds}\label{sec-application}

In this final section, 
we prove Theorem \ref{thm_mainSchoen} and Corollary~\ref{cor_ratpol}
about Schoen threefolds.

\subsection{Decomposition of effective cones of fiber products} 

 \begin{proposition}\label{prop_decom_eff} For $i=1,2$, let $\phi_i : W_i \to \PP^1$ be a surjective morphism from a projective variety to $\PP^1$. Assume that
    	\begin{enumerate}
    		\item  the variety $W=W_1 \times_{\PP^1} W_2$ is irreducible; 
    		\item  we have $$p_1^*N^1(W_1)_{\RR}+p_2^*N^1(W_2)_{\RR}=N^1(W)_{\RR}$$ 
      where $p_i : W \to W_i$ are the natural projections. 
    	\end{enumerate}
    	Then 
    	$$p_1^*\Eff(W_1) + p_2^*\Eff(W_2)=\Eff(W).$$ 
    \end{proposition}
    
    \begin{proof} 
    The proof is similar to Namikawa's argument~\cite[p. 153]{Na91}. 
    We draw the commutative diagram here for reader's convenience: 
    \[ \xymatrix@=1.5em{ & W \cnec W_1 \times_{\PP^1} W_2 \ar[dl]_{p_1} \ar[dd]_{\phi} \ar[dr]^{p_2}   \\
	W_1 \ar[dr]_{\phi_1} & &     W_2 \ . \ar[dl]^{\phi_2}   \\
	& \PP^1 } \]
    Take an arbitrary integral element in $\Eff(X)$,  
    represented by an effective line bundle $L$. 
    Since $p_i^{\ast}$ is defined over $\QQ$, (2) implies that up to replacing $L$ by a positive multiple of it, we have $L = p_1^{\ast}L_1\otimes p_2^{\ast}L_2$ with $L_i\in \Pic(W_i)$. Then 
    \begin{align*}
        \phi_{\ast}L &= \phi_{1\ast}p_{1\ast} \left( p_1^{\ast}L_1\otimes p_2^{\ast}L_2 \right) \\
        &= \phi_{1\ast} \left( L_1\otimes p_{1\ast}(p_2^{\ast}L_2) \right) \\
        &= \phi_{1\ast} \left( L_1\otimes \phi_1^{\ast}(\phi_{2\ast}L_2) \right) \\
        &= \phi_{1\ast} L_1\otimes \phi_{2\ast}L_2, 
    \end{align*}
    where the second and the last equalities follow from the projection formula, and the third equality follows from the flat base change theorem. Note that $\phi_{\ast}L$ has a non-zero section as $L$ is effective. 
    Moreover, ${\phi_i}_{\ast}L_i$ as a vector bundles 
    on $\PP^1$ is a direct sum of line bundles. So there are summands $\OO_{\PP^1}(a_i)$ of $\phi_{i\ast}L_i$ such that $a_1 + a_2 \geq 0$. Without loss of generality, we may assume $a_1 \geq 0$. Then both $\phi_{1\ast}L_1 \otimes \OO_{\PP^1}(-a_1)$ and $\phi_{2\ast}L_2 \otimes \OO_{\PP^1}(a_1)$ have non-zero sections. Now let $M_1:= L_1 \otimes \phi_1^{\ast}\OO_{\PP^1}(-a_1)$ and $M_2:= L_2 \otimes \phi_2^{\ast}\OO_{\PP^1}(a_1)$. Then $L = p_1^{\ast}M_1\otimes p_2^{\ast}M_2$, and both $M_i$ are effective as $\phi_{i\ast}M_i$ have non-zero sections. 
    \end{proof}

\ssec{Proof of Theorem \ref{thm_mainSchoen}} 

We start by proving the following lemma.

\begin{lemma}\label{lem-dimZ2}
Let $\phi: W \to \PP^1$ be a relatively minimal rational elliptic surface with a section. There exists a polyhedral cone $\Pi\subset\Eff(W)$ such that
$$\Aut(W/\PP^1)\cdot \Pi = \Eff(W)$$
where $\Aut(W/\PP^1) := \left\{ f \in \Aut(W) : \phi\circ f = \phi \right\}.$
\end{lemma}

\begin{proof} This is essentially a consequence of Totaro's results \cite{To10}. We still provide a proof, because our setting is slightly different than Totaro's.
We take smooth fibers $F_1,\ldots,F_4$ of $\phi$ above four distinct points $P_1,\ldots,P_4$ of $\PP^1$ which are general enough that $\Aut(\PP^1,P_1+\ldots+P_4)$ is trivial. 

We set 
$$\Delta\cnec \frac{1}{4}F_1 + \cdots + \frac{1}{4}F_4.
$$
Then $\Supp \Delta$ is simple normal crossing, because these $F_i$'s are smooth and pairwise disjoint.  So the pair $(W, \Delta)$ is klt. 
By~\cite[Proposition 4.9.3(1)]{EnquesBookI}, 
$-K_W \sim F_i$ for each $i$, and thus, $K_W + \Delta \sim_{\QQ} 0$. This shows that $(W, \Delta)$ is Calabi--Yau. Now we can apply \cite[Theorem 4.1]{To10} to conclude that the pair $(W, \Delta)$ satisfies the nef cone conjecture.
Hence, by Theorem \ref{main_thm-equiv} (and since any small $\QQ$-factorial modification is an isomorphism, in dimension 2), this pair also satisfies the effective cone conjecture, i.e., there is a rational polyhedral cone
$\Pi\subset\Eff(W)$ such that
$$\Aut\left(W,\Delta\right)\cdot\Pi = \Eff(W).$$
By our choice of $P_1,\ldots, P_4$, we have that $\Aut(W/\PP^1) = \Aut\left(W,\Delta\right)$, which concludes this proof. 
\end{proof}

We can now proceed to the proof of Theorem \ref{thm_mainSchoen}.

\begin{proof}[Proof of Theorem \ref{thm_mainSchoen}]
   We start by recalling the setting. Let $X$ be a smooth Calabi--Yau threefold obtained as a fiber product  $W_1\times_{\PP^1}W_2$, where $\phi_i\colon W_i\to\PP^1$ is a relatively minimal rational elliptic surface with a section for $i=1,2$.
    
    We assume that the generic fibers of $\phi_1$ and $\phi_2$ are non-isogenous. Hence, by \cite[Lemma 5.1]{GLW22} (see also \cite[Proof of Proposition 1.1]{Na91} for a proof under the additional assumption that all singular fibers of the $\phi_i$ are of type $I_1$), we have a decomposition
$$p_1^*N^1(W_1)_{\RR}+p_2^*N^1(W_2)_{\RR}=N^1(X)_{\RR},$$
where each $p_i:X\to W_i$ is the projection.

    By Proposition \ref{prop_decom_eff}, we have
$p_1^*\Eff(W_1) + p_2^*\Eff(W_2) = \Eff(X).$ By Lemma \ref{lem-dimZ2}, there exists a polyhedral cone $\Pi_i \subset \Eff(W_i)$ such that $\Aut(W_i/\PP^1) \cdot \Pi_i = \Eff(W_i)$. Let us introduce the polyhedral cone $\Sigma \subset \Eff(X)$ generated by $p_1^{\ast}\Pi_1$ and $p_2^{\ast}\Pi_2$. Since $\Aut(X)$ contains $\Aut(W_1/\PP^1)\times\Aut(W_2/\PP^1)$, this yields $$\Aut(X)\cdot \Sigma = \Eff(X).$$ 
Thus, by Proposition \ref{pro-looij}, the effective cone conjecture holds for $X$. Since good minimal models exist in dimension $3$, the movable cone conjecture, as well as all of the statements (1) to (4) in Theorem~\ref{main_thm-equiv} holds for $X$. 
\end{proof}

\begin{proof}[Proof of Corollary \ref{cor_ratpol}]
The assumptions of Corollary \ref{cor_ratpol} are the same as in~\cite[p. 152]{Na91}. Hence,~\cite[p. 162]{Na91} shows that any minimal model $X'$ of $X$ with $X'\not\cong X$ has finite automorphism group. 
By Theorem \ref{thm_mainSchoen}, the nef cone conjecture holds for $X'$.  Hence $\Nef(X')$, as a union of finitely many rational polyhedral cones, is itself rational polyhedral. 
\end{proof}

\begin{remark} We may also allow $W_i$ to be a weak del Pezzo surface with a fibration $W_i \to \PP^1$. Take the fiber product $W:= W_1\times_{\PP^1} W_2$ and assume that $W$ is smooth. Then we can get a klt Calabi--Yau pair $(W, \Delta)$ for a suitable $\Delta$, and Statements (1) to (4) in Theorem~\ref{main_thm-equiv} 
also hold for $(X, \Delta)$. 
\end{remark}

\bibliographystyle{alpha}
\bibliography{Mov}

\end{document}